\documentclass{amsproc}

\usepackage{amsmath,amssymb}
\usepackage{graphicx}
\newtheorem{theorem}{Theorem}
\newtheorem{defi}[theorem]{Definition}
\newtheorem{lemma}[theorem]{Lemma}
\newtheorem{coro}[theorem]{Corollary}
\newtheorem{proposition}[theorem]{Proposition}
\newtheorem{remark}[theorem]{Remark}
\newtheorem{example}[theorem]{Example}
\usepackage[a4paper, margin=1in]{geometry}

\begin{document}
	
	\title[Harish-Chandra theorem for two-parameter quantum groups]{Harish-Chandra Theorem for Two-parameter \\ Quantum Groups}

	\author[N.H. Hu]{Naihong Hu$^*$}
	\address{School of Mathematical Sciences,  MOE Key Laboratory of Mathematics and Engineering Applications \& Shanghai Key Laboratory of PMMP, East China Normal University, Shanghai 200241, China}
	\email{nhhu@math.ecnu.edu.cn}
	
	\author[H.Y. Wang]{Hengyi Wang}
	\address{School of Mathematical Sciences, East China Normal University, Shanghai 200241, China}
	\email{52265500001@stu.ecnu.edu.cn}

	\thanks{This work is
		supported by the NNSF of China (Grant No. 12171155), and in part by the Science and Technology Commission of Shanghai Municipality (Grant No. 22DZ2229014), the Research Fund of Jianghan Univ. (No. 2023JCYJ08).}

	\date{2024.12.24}
	
	\begin{abstract}
		
	This paper is devoted to investigating the centre of two-parameter quantum groups $U_{r,s}(\mathfrak{g})$ via establishing the Harish-Chandra homomorphism. Based on the Rosso form and the representation theory of weight modules, we prove that when rank $\mathfrak{g}$ is even, the Harish-Chandra homomorphism is an isomorphism, and in particular, the centre of the quantum group $\breve{U}_{r,s}(\mathfrak{g})$ of the weight lattice type is a polynomial algebra $\mathbb{K}[z_{\varpi_1},\cdots,z_{\varpi_n}]$, where canonical central elements $z_\lambda \; (\lambda \in \Lambda^+)$ are turned out to be uniformly expressed. For rank $\mathfrak{g}$ to be odd, we figure out a new invertible extra central generator $z_*$, which doesn't survive in $U_q(\mathfrak g)$, then the centre of $\breve{U}_{r,s}(\mathfrak{g})$ contains $\mathbb{K}[z_{\varpi_1},\cdots,z_{\varpi_n}]\otimes_\mathbb K\mathbb K[z_*^{\frac{1}{\ell}}, z_*^{-\frac{1}{\ell}}]$, where $\ell=2$, except  $\ell=4$ for $D_{2k+1}$.
		
	\end{abstract}
	
	\maketitle

\section{Introduction}

In mathematics and physics, quantum groups, in a broad sense, are a class of non-commutative and non-cocommutative Hopf algebras. It was Drinfel'd  \cite{D86} and Jimbo \cite{J86} who  independently defined $U_q(\mathfrak g)$ as a $q$-deformation of the universal enveloping algebra $U(\mathfrak g)$ for any semisimple Lie algebra $\mathfrak g$. Such $q$-deformed objects provide the universal solutions for the quantum Yang-Baxter equation and numerous quantum invariants for knots or links even $3$-manifolds (e.g. \cite{Re90, ZGB91, Turaev94} and references therein).

A number of works on the centre of quantum groups $U_q(\mathfrak g)$ have been developed over the last three decades, including explicit generators and structures.
Rosso introduced quantum analogues of the Killing form and the Harish-Chandra map to investigate the structure of the centre $Z(U_q)$ \cite{Ro90}. Drinfel'd used the universal R-matrix to get a morphism from the representation ring of $\mathfrak{g}$ to $Z(U_q)$ \cite{D90} (see also the works of Reshetikhin et al. \cite{FRT88, Re90}). Subsequently, Tanisaki \cite{T90,T92}, as well as Joseph and Letzter \cite{JL92}, proved Harish-Chandra isomorphism theorem in different approaches. Connections between these constructions have been discussed in various works in the literature (e.g., see Baumann's \cite{B97}). A couple of years ago, Li-Xia-Zhang demonstrated that $Z(U_q)$ is isomorphic to a polynomial algebra in types $A_1,$ $B_n,$ $C_n,$ $D_{2n},$ $E_7,$ $E_8,$ $F_4,$ $G_2$, while in the remaining cases it is isomorphic to a quotient of polynomial algebra \cite{LXZ16, LXZ18}.
A general description of the centres of quantum groups $\breve{U}_q(\mathfrak{g})$ of weight lattice types was earlier considered by Etingof in \cite{E95}. He concluded that $Z(\breve{U}_q) \cong \mathbb{K}[c_{\varpi_1},\cdots,c_{\varpi_n}]$,
where $c_{\varpi_i}$ represents  the quantum partial trace of $\Gamma = R^{21}R$ on $L(\varpi_i)$ (from Reshetikhin et al.).
The theorem was in fact generalized to the affine cases \cite{E95}.
Dai supplemented a detailed proof for the theorem of $Z(\breve{U}_q)$ and an explicit formula of generator $c_\lambda$ \cite{D23} via an operator $\Gamma$ from the work of R.B. Zhang et al \cite{GZB91,ZGB91}. Recently, Luo-Wang-Ye settled the Harish-Chandra theorem for some quantum superalgebras \cite{LWY22}.

\begin{figure}[htbp]
	\includegraphics[height=3.6cm]{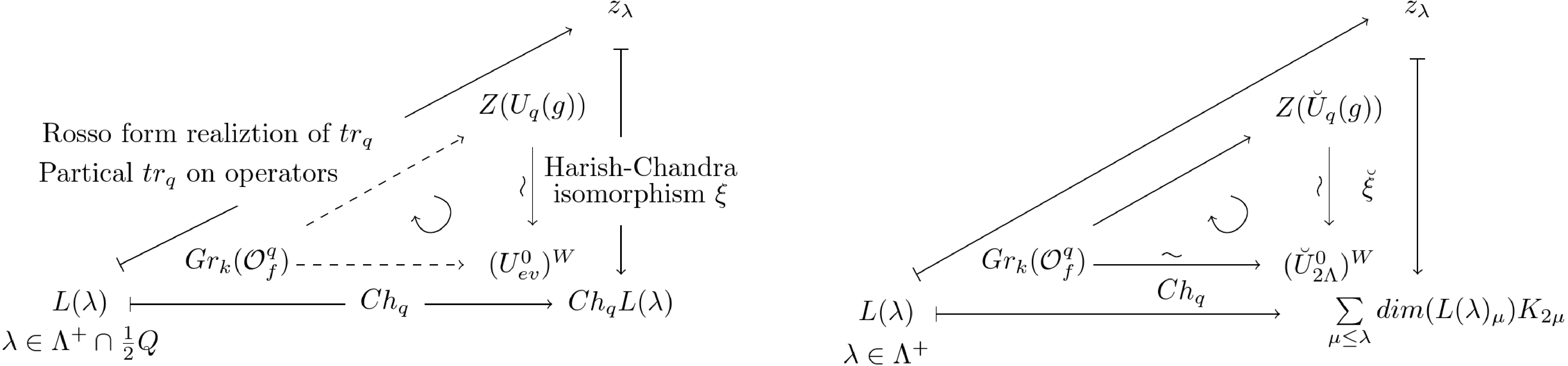}
	\caption{The central structure of the one-parameter quantum group.}
\end{figure}

Since Benkart-Witherspoon redefined a class of two-parameter quantum groups of type $A$ motivated by up-down algebra \cite{BW04}, and Bergeron-Gao-Hu studied the structures of two-parameter quantum groups of type $B, C, D$ \cite{BGH06}, the representation theory has progressed simultaneously \cite{BW04,BGH07}.
A series of work on other types have been done in \cite{BH08,HS07,HW09,PHR10,CHW16}, etc.

In this paper, we focus on the description of the centres of two-parameter quantum groups $U_{r,s}(\mathfrak{g})$ in the case when $r,s$ are indeterminates. We give a uniform way to prove that the Harish-Chandra homomorphism $\xi$ of $U_{r,s}(\mathfrak{g})$ is an isomorphism when $\mathfrak{g}$ is simple with even rank. This in particular recovers the work of Benkart-Kang-Lee for type $A_{2n}$ \cite{BKL06}, Hu-Shi for type $B_{2n}$ \cite{HS14} and Gan for type $G_2$ \cite{G10}. To do this, there are 3 steps to do: (i) $\xi : Z(U) \longrightarrow U^0$ is injective. (ii) The image $\text{Im}(\xi)$ is in the subalgebra $(U^0_\flat)^W$. (iii)  $\xi : Z(U) \longrightarrow  (U^0_\flat)^W$ is surjective.
With the help of the weight module theory established by \cite{BGH07, HP12, PHR10}, we find that step (i) and (ii) only rely on the non-degeneracy of matrices $R-S$ and $R+S$ derived from the  structure constant matrix.
Then we construct a certain element $z_\lambda \; (\lambda \in \Lambda^+ \cap Q)$ to realize the quantum trace $\operatorname{tr}_{r,s}$ on the weight module $L(\lambda)$ by the Rosso form, that is, $\langle z_\lambda \;|\; -\rangle = \operatorname{tr}_{L(\lambda)} (- \circ \Theta)$. These $z_\lambda$ are central elements which are used to prove (iii). Although we successfully establish the Harish-Chandra isomorphism (see the vertical map in the first diagram of Fig. 2),
owing to the fact that the Rosso form realization of $\operatorname{tr}_{r,s}$ and the quantum character both require $\lambda \in \Lambda^+ \cap Q$ (this is distinct from the one-parameter case: $\lambda \in \Lambda^+ \cap \frac{1}{2} Q$ ), they are not always well-defined on the whole $\operatorname{Gr}_k (\mathcal{O}^{r,s}_f) := \operatorname{Gr}(O^{r,s}_f) \otimes_{\mathbb{Z}} \mathbb{K}$.
\begin{figure}[htbp]
	\includegraphics[height=3.56cm]{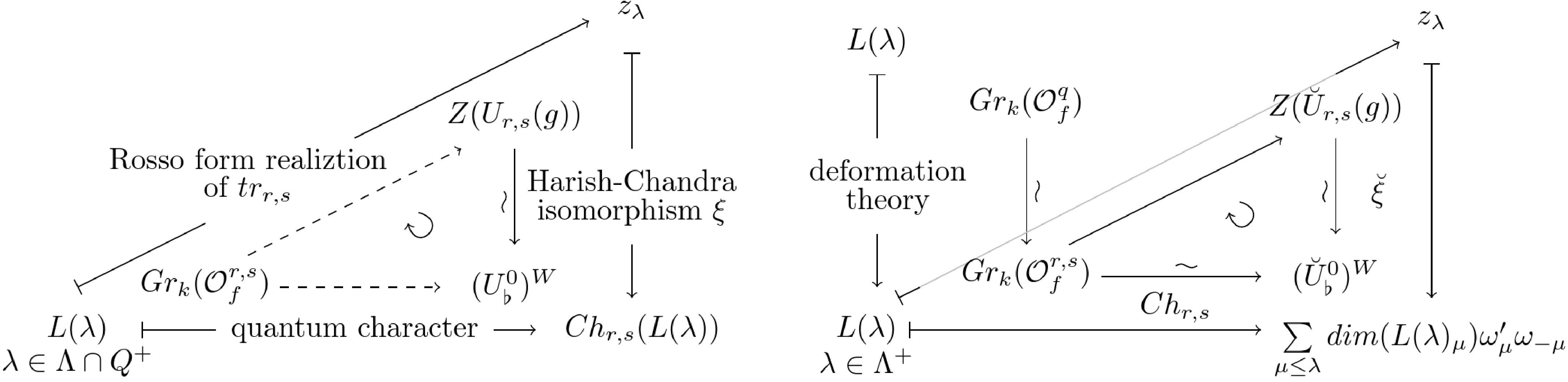}
	\caption{The central structure and Harish-Chandra theorem of $U_{r,s}(\mathfrak{g})$ with even rank.}
\end{figure}
To overcome this difficulty, we need to extend $U_{r,s}(\mathfrak{g})$ to its weight lattice type $\breve{U}_{r,s}(\mathfrak{g})$ such that all the maps above are well-defined (see the second diagram of Fig. 2). It turns out that the Harish-Chandra homomorphism $\breve{\xi}$ is also an isomorphism. Furthermore, the bottom map $\operatorname{Ch}_{r,s}$ is shown to be an isomorphism of algebras. Finally, the deformation theory on weight modules \cite{HP12} tells us that our $\mathcal{O}^{r,s}_f$  is equivalent to  the category of finite-dimensional weight modules $\mathcal{O}^q_f$ of $U_q(\mathfrak{g})$ as braided tensor categories. Hence the centre $Z(\breve{U}_{r,s}) \cong (\breve{U}^0_\flat)^W \cong \operatorname{Gr}_k (O^{r,s}_f) \cong \operatorname{Gr}_k (O^{q}_f)$, which is also a polynomial algebra as in the one-parameter setting, that is, $Z(\breve{U}_{r,s}) = \mathbb{K}[z_{\varpi_1},\cdots,z_{\varpi_n}]$.

In the case of odd rank, the Harish-Chandra homomorphism $\xi$ remains injective, but its image $\operatorname{Im}(\xi)$ cannot be explicitly characterized since the matrix $R+S$ has corank $1$. It provides a unique invertible central generator $z_*$ (it will degenerate to $1$ in the one-parameter case, i.e., $r=q,\, s=q^{-1},\, \omega'_i = \omega_i^{-1}$), which is a fixed point of $\xi$ satisfying $z_* \notin (U^0_\flat)^W$. This gives $\operatorname{Im}(\xi) \supseteq (U^0_\flat)^W \otimes \mathbb{K}[z_{*},z_{*}^{-1}]$ for $U_{r,s}(\mathfrak{g})$, which recovers the case of type $A_{2n+1}$ \cite{BKL06}, and we have $Z(\breve{U}_{r,s}(\mathfrak{g})) \supseteq \mathbb{K}[z_{\varpi_1},\cdots,z_{\varpi_n}]\otimes_\mathbb K \mathbb K[z_*^{\frac{1}{\ell}}, z_*^{-\frac{1}{\ell}}]$, where $\ell=2$, except $\ell=4$ for $D_{2k+1}$.

The paper is structured as follows.
Section 2 recalls two-parameter quantum groups with their Rosso forms, the Harish-Chandra homomorphisms and results in the related literature. Then we introduce two matrices relevant to the structure constant matrix.
Section 3 proves that the Harish-Chandra homomorphism $\xi$ is injective.
Section 4 characterizes the image of $\xi$ and proves that $\operatorname{Im}(\xi)$ falls in the subalgebra $(U^0_\flat)^W$ when rank $n$ is even.
Section 5 proves that the Harish-Chandra image $\xi(Z(U))\supseteq (U^0_\flat)^W$. Then the Harish-Chandra theorem $\xi:Z(U) \cong (U^0_\flat)^W$ holds when rank $n$ is even. While in the odd rank case, we construct a new extra generator $z_*$ of the centre, which leads to the Harish-Chandra image $\xi(Z(U)) \supseteq (U^0_\flat)^W \otimes \mathbb{K}[z_*^{\pm 1}]$.
We also give an alternative description to the central elements by taking partial quantum trace.
The last section introduces the extended two-parameter quantum group  $\breve{U}_{r,s}(\mathfrak{g})$ of weight lattice type and proves the Harish-Chandra theorem $\breve{\xi}: Z(\breve{U}_{r,s}(\mathfrak{g})) \cong (\breve{U}^0_\flat)^W$ when rank $n$ is even. Particularly, in this case the centre $Z(\breve{U}) = \mathbb{K}[z_{\varpi_1},\cdots,z_{\varpi_n}]$ is a polynomial algebra. While for the cases of odd rank $n$, the centre of $\breve{U}_{r,s}(\mathfrak{g})$ contains $\mathbb{K}[z_{\varpi_1},\cdots,z_{\varpi_n}]\otimes_\mathbb K \mathbb K[z_*^{\frac{1}{\ell}}, z_*^{-\frac{1}{\ell}}]$.

\section{Preliminaries}

\subsection{The two-parameter quantum group $U_{r,s}(\mathfrak{g})$ and its Hopf algebra structure}

A root system $\Phi$ of one complex simple Lie algebra $\mathfrak{g}$ is a finite subset of a Euclidean space $E$. We fix a simple root set of $\Phi$ and denote it by $\Pi.$ Let $W$ be the Weyl group of the root system $\Phi$ and $\sigma_i \in W$ be the simple reflection corresponding to the simple root $\alpha_i$. For data of $\Phi$ and $\Pi$, see Carter's book \cite[p.543]{C05}. We normalize the inner product $(-,-)$ such that $(\alpha,\alpha) = 2$ for all short roots $\alpha$. Let $C =(c_{ij})_{n\times n}$ be the Cartan matrix, where $c_{ij} = \frac{2(\alpha_i,\alpha_j)}{(\alpha_i,\alpha_i)} . $ 	

Denote $Q = \bigoplus_{i=1}^n \mathbb{Z} \alpha_i$ as the root lattice of $\mathfrak{g}.$	Put $\alpha_i^\vee = \frac{2\alpha_i}{(\alpha_i,\alpha_i)}$, let $\Lambda = \bigoplus_{i=1}^n \mathbb{Z} \varpi_i $ be the weight lattice and $\Lambda^+ = \bigoplus_{i=1}^n \mathbb{Z}_+ \varpi_i $ be the set of dominant integral weights, where $\varpi_i = \sum_{k=1}^n {c}^{ki} \alpha_{k}$  is the $i$-th fundamental weight, and $ ({c}^{ij})_{n\times n} = C^{-1}$ is the inverse of the Cartan matrix.

Let $r,s$ be two indeterminates and $\mathbb{Q}(r,s)$ be the rational functions field.
Let $\mathbb{K} \supseteq \mathbb{Q}(r,s)$ be a field that contains $r^{\frac{1}{m^2}},s^{\frac{1}{m^2}}$, where $m = \operatorname{min}\{k\in \mathbb{Z}^+ \mid k\Lambda \subseteq Q\} = \operatorname{det} (C).$
Let $r_i = r^{d_i} = r^{\frac{(\alpha_i,\alpha_i)}{2}}$ and $ s_i = s^{d_i}= s^{\frac{(\alpha_i,\alpha_i)}{2}}$ for $i=1,\dots,n$.
Let $A = (a_{ij})_{n\times n}$ be the structure constant matrix of $U_{r,s}(\mathfrak{g})$, which is given by the Euler form of $\mathfrak{g}$ as follows.

\begin{defi}\cite{HP08} \label{Euler}
	The Euler form of $\mathfrak{g}$ is the bilinear form $\langle - , -\rangle$ defined on the root lattice $Q$ satisfying
	\begin{align*}
		\langle i,j \rangle := \langle \alpha_i,\alpha_j \rangle = \begin{cases}
			& d_i c_{ij}  \hspace{1cm} i<j,\\
			& d_i         \hspace{1.4cm} i=j,\\			
			& 0           \hspace{1.6cm} i>j.			
		\end{cases}
	\end{align*}
	For type D, it is necessary to revise $\langle n-1,n \rangle = -1, \, \langle n, n-1 \rangle = 1$ $($\cite{BGH06,HP12}$)$.
	
	It can be linearly extended to the weight lattice $\Lambda$ such that
	\begin{align*}		
		\langle \varpi_i, \varpi_j \rangle := \sum_{k,l=1}^{n} c^{ki}c^{lj}\langle k,l \rangle
		.		\end{align*}		
\end{defi}

Write $a_{ij}= r^{\langle j,i \rangle} s^{- \langle i,j \rangle}$, and $A = (a_{ij})_{n\times n}$ is the structure constant matrix of $U_{r,s} (\mathfrak{g})$. Denote $ R_{ij} = \langle j,i \rangle$, $S_{ij}= - \langle i,j \rangle.$ Then we have $a_{ij} = r^{R_{ij}} s^{S_{ij}}$ and $R = -S^T.$

\begin{defi}
	\textup{(cf., \cite{BGH06,HP08})} \label{def} 		
	Let $U = U_{r,s} (\mathfrak{g})$ be the unital associative algebra over $\mathbb{K}$, generated by elements $e_i, f_i, \omega_i^{\pm 1}, {\omega'_i}^{\pm 1}\; (i=1,\cdots,n)$ satisfying the following relations $(X1)$---$(X4):$
	
	\hspace{1cm}
	
	$(X1)$ \, $\omega_i^{\pm 1} {\omega'_j}^{\pm 1} =  {\omega'_j}^{\pm 1}\omega_i^{\pm 1},  \qquad \omega_i \omega_i^{-1} = 1= \omega'_j {\omega'_j}^{-1},$
	
	$(X2)$ \, $\omega_{i} e_{j} \omega_{i}^{-1}= a_{ij} e_{j}, \hspace{0.7cm} \qquad \omega_{i} f_{j} \omega_{i}^{-1} ={a_{ij}}^{-1} f_{j},$
	
	\hspace{0.9cm} 	$\omega_{i}^{\prime} e_{j} \omega_{i}^{\prime-1} ={a_{ji}}^{-1} e_{j},  \hspace{0.3cm} \qquad
	\omega_{i}^{\prime} f_{j} \omega_{i}^{\prime-1} =a_{ji} f_{j},$
	
	$(X3)$ \, $[e_i, f_j] = \delta_{ij} \frac{\omega_i -\omega_i'}{r_i - s_i},$
	
	$(X4)$ \, $(\operatorname{ad}_l e_i)^{1-c_{ij}}(e_j) = 0,
	\qquad (\operatorname{ad}_r f_i)^{1-c_{ij}}(f_j) = 0,\qquad (i \neq j)$	
	
	\bigskip
	\noindent where the left (right)-adjoint action are as follows: for any $x,y \in U_{r,s} (\mathfrak{g}),$
	\begin{align*}
		\operatorname{ad}_l x (y) = \sum\limits_{(x)} x_{(1)} y S(x_{(2)}), & & \operatorname{ad}_r 	x (y) = \sum\limits_{(x)}S(x_{(1)}) y x_{(2)}.
	\end{align*}
	The comultiplication $\Delta(x) = \sum_{(x)} x_{(1)} \otimes x_{(2)}$ is given by Proposition \ref{Hopf} below.
\end{defi}

\begin{proposition}\textup{(cf., \cite{BGH06,HP08})} \label{Hopf}
	The algebra $U_{r,s} (\mathfrak{g})$ has a Hopf algebra structure $(U_{r,s} (\mathfrak{g}),$ $\Delta,\; \varepsilon,\; M,\; \iota,\; S ) $ with the comultiplication $\Delta$, the counit $\epsilon$ and the antipode $S$ defined below:
	\begin{align*}
		&\Delta(\omega_i^{\pm 1}) = \omega_i^{\pm 1} \otimes \omega_i^{\pm 1}, \hspace{2.5cm}
		\Delta({\omega'_i}^{\pm 1}) = {\omega'_i}^{\pm 1} \otimes {\omega'_i}^{\pm 1},\\
		&\Delta(e_i) = e_i \otimes 1 + \omega_i \otimes e_i \;, \hspace{2.035cm}
		\Delta(f_i) = 1 \otimes f_i + f_i \otimes {\omega'_i} \;,\\
		&\varepsilon(\omega_i^{\pm 1}) = \varepsilon({\omega'_i}^{\pm 1}) = 1, \hspace{2.5cm}
		\varepsilon(e_i) = \varepsilon(f_i) = 0,\\
		&S(\omega_i^{\pm 1}) = \omega_i^{\mp 1}, \hspace{3.65cm}
		S({\omega'_i}^{\pm 1}) = {\omega'_i}^{\mp 1} ,\\
		&S(e_i) = -\omega_i^{-1} e_i \;, \hspace{3.35cm}
		S(f_i) = -f_i {\omega'_i}^{-1}.
	\end{align*}
\end{proposition}

The two-parameter quantum group $U=U_{r,s}(\mathfrak{g})$ has the triangular decomposition $U \cong U^- \otimes U^0 \otimes U^+$, where $U^0$ is the subalgebra generated by $\{\omega_i^{\pm}, {\omega'_i}^{\pm}, 1 \leqslant i \leqslant n\}$, $U^+$ is generated by $\{e_i,1 \leqslant i \leqslant n\}$ and $U^-$ is generated by $\{f_i,1 \leqslant i \leqslant n\}.$
Let $\mathcal{B}$ be the subalgebra of $U$ generated by $\{e_j,\omega_j^{\pm}, 1\leqslant j \leqslant n\}$, and $\mathcal{B}'$ the subalgebra of $U$ generated by $\{f_j,{\omega'_j}^{\pm}, 1\leqslant j \leqslant n\}$.

\begin{proposition}
	\textup{(cf., \cite{BGH06,HP08})}	\label{skew-pair}
	There exists a unique skew-dual pairing $\langle  -,- \, \rangle: \mathcal{B}' \times \mathcal{B} \longrightarrow \mathbb{K}$ of the Hopf subalgebras $\mathcal{B}$ and $\mathcal{B}'$ satisfying
	\begin{gather*}
		\langle f_i , e_j \rangle = \delta_{ij} \frac{1}{s_i - r_i }, \\
		\langle \omega'_i , \omega_j \rangle = a_{ji},\\
		\langle {\omega'_i}^{\pm 1} , \omega_j^{-1} \rangle = \langle {\omega'_i}^{\pm 1} , \omega_j \rangle ^{-1} = \langle \omega'_i , \omega_j \rangle^{\mp 1},
	\end{gather*}
	for  $1 \leqslant i , j \leqslant n$, and all other pairs of generators are $0$.
	Moreover, $\langle S(a),S(b) \rangle = \langle a,b \rangle $ for $ a \in \mathcal{B}', b \in \mathcal{B} \;.$
\end{proposition}

For $\eta = \sum_{i=1}^{n} \eta_i \alpha_i \in Q$, write
\[\omega_\eta = \omega_1^{\eta_1} \cdots \omega_n^{\eta_n} , \hspace{2cm} \omega'_\eta = {\omega'_1}^{\eta_1} \cdots {\omega'_n}^{\eta_n}.\]
Then we have
\begin{align*}
	& \omega_\eta e_i \omega_\eta^{-1} = \langle \omega'_i, \omega_\eta \rangle e_i,   \hspace{2cm}
	\omega_\eta f_i \omega_\eta^{-1} = \langle \omega'_i, \omega_\eta \rangle^{-1} f_i, \\
	& \omega'_\eta e_i {\omega'_\eta}^{-1} = \langle \omega'_\eta, \omega_i \rangle^{-1} e_i,   \hspace{1.7cm}
	\omega'_\eta f_i {\omega'_\eta}^{-1} = \langle \omega'_\eta, \omega_i \rangle f_i.
\end{align*}
Introduce a $Q$-graded structure on $U_{r,s}(\mathfrak{g}):$
\[\operatorname{deg}\; e_i = \alpha_{i}, \hspace{1cm} \operatorname{deg}\; f_i = -\alpha_{i}, \hspace{1cm} \operatorname{deg}\; \omega_i = \operatorname{deg}\; \omega'_i = 0. \]
Namely, $U_{r,s}(\mathfrak{g})$ is a $Q$-graded algebra, and its subalgebras $U^\pm$ are $Q^\pm$-graded.
\begin{align*}
	U^+ = \bigoplus_{\mu \in Q^+} U^+_\mu, \qquad U^-=\bigoplus_{\mu \in Q^+} U^-_{-\mu},
\end{align*}
where $U^+_\mu = U^+ \cap U_\mu,\; U^-_{-\mu} = U^- \cap U_{-\mu}$ and
\[U_\mu^{\pm} = \{x \in U^{\pm} \;|\; \omega_\eta x \omega_\eta^{-1} = \langle \omega'_\mu , \omega_\eta \rangle x,\; \omega'_\eta x {\omega'_\eta}^{-1} = \langle \omega'_\eta, \omega_\mu \rangle^{-1} x \}.\]
And $\mathcal{B}$ (resp. $\mathcal{B}'$) is also $Q^+$(resp. $Q^-$)-graded algebras.

\subsection{The matrices $R$ and $S$}

\begin{proposition} \label{sym}
	The matrix $R-S$ is a symmetric Cartan matrix, i.e.
	\begin{align*}
		R-S = DC = \big((\alpha_{i},\alpha_{j})\big)_{n\times n}	
	\end{align*}
	where the matrix $D = \operatorname{diag}(d_1,\cdots,d_n) ,\; d_i = (\alpha_{i},\alpha_{i})/2.$
\end{proposition}
\begin{proof} By definition, we have
	$R-S  = R + R^T = \left(\langle j,i \rangle + \langle i, j \rangle \right)_{n\times n} = DC$.
\end{proof}

The symmetric matrix $R-S$ can be considered as a metric matrix with respect to the basis $\alpha_1,\cdots,\alpha_n$ of $E$, and every element of Weyl group is orthogonal, we then have

\begin{coro} \label{InvR-S} Let $\Sigma$ be the matrix of $\sigma \in W$ with respect to the basis $\{\alpha_i\}$, we have $(R-S)\Sigma = \Sigma^T (R-S). $
\end{coro}

Now we will discuss the non-degeneracy of another important matrix $R+S$ which is crucial in the study of the Harish-Chandra theorem of $U_{r,s} (\mathfrak{g})$.
\begin{proposition} \label{InvR+S}
	When $n$ is even, the matrix $R+S$ is invertible;
	when $n$ is odd, we have $\operatorname{rank}(R+S)= n-1.$
\end{proposition}	
\begin{proof}
	We check the conclusion case by case.
	For type $A_n$, we have
	$$ \operatorname{ det } (R+S) = \begin{cases}
		1, & 2 \mid  n, \\
		0, & 2 \nmid n,
	\end{cases}$$
	where
	\begin{equation*}\resizebox{1\hsize}{!}{$
			A =
			\left(\begin{array}{ccccccc}
				rs^{-1} &    s     &    1     &    \cdots    &    1    &    1    &   1      \\
				r^{-1}  & rs^{-1}  &    s     &    \cdots    &    1    &    1    &   1      \\
				1       & r^{-1}   & rs^{-1}  &    \cdots    &    1    &    1    &   1      \\
				\vdots  &  \vdots  &  \vdots  &    \ddots    & \vdots  & \vdots  & \vdots   \\
				1       &    1     &    1     &    \cdots    & rs^{-1} &    s    &   1      \\
				1       &    1     &    1     &    \cdots    &  r^{-1} & rs^{-1} &   s      \\
				1       &    1     &    1     &    \cdots    &    1    &  r^{-1} & rs^{-1}
			\end{array}\right), \hspace{0.1cm}
			R+S =
			\left(\begin{array}{cccccc}
				&   1      &              &         &         &          \\
				-1  &          &    \ddots    &         &         &          \\
				& \ddots   &              &  \ddots &         &          \\
				&          &    \ddots    &         &    1    &          \\
				&          &              &    -1   &         &   1      \\
				&          &              &         &   -1    &
			\end{array}\right).
			$}
	\end{equation*}
	
	For type $B_n$, we have
	$$\operatorname{ det } (R+S) = \begin{cases}
		2^{n}, & 2 \mid  n, \\
		0,       & 2 \nmid n,
	\end{cases}$$
	where
	\begin{equation*}\resizebox{1\hsize}{!}{$
			A =
			\left(\begin{array}{cccccc}
				r^2 s^{-2} &    s^2      &    \cdots    &    1       &    1       &   1      \\
				r^{-2}     & r^2 s^{-2}  &    \ddots    &    1       &    1       &   1      \\
				\vdots     &  \ddots     &    \ddots    & \ddots     & \vdots     & \vdots   \\
				1          &    1        &    \ddots    & r^2 s^{-2} &  s^2       &   1      \\
				1          &    1        &    \cdots    &  r^{-2}    & r^2 s^{-2} &   s^2    \\
				1          &    1        &    \cdots    &    1       &   r^{-2}   & rs^{-1}
			\end{array}\right),  \hspace{0.1cm}
			R+S  =
			\left(\begin{array}{cccccc}
				&   2      &              &         &         &          \\
				-2  &          &  \ddots      &         &         &          \\
				&  \ddots  &              & \ddots  &         &          \\
				&          &    \ddots    &         &    2    &          \\
				&          &              &    -2   &         &   2      \\
				&          &             &         &   -2    &
			\end{array}\right).
			$}
	\end{equation*}
	
	For type $C_n$, we have
	$$\operatorname{ det } (R+S) = \begin{cases}
		4,       & 2 \mid  n, \\
		0,       & 2 \nmid n,
	\end{cases}$$
	where
	\begin{equation*}\resizebox{1\hsize}{!}{$
			A  =
			\left(\begin{array}{cccccc}
				r   s^{-1} &    s        &    \cdots    &    1       &    1       &   1      \\
				r^{-1}     &   r s^{-1}  &    \ddots    &    1       &    1       &   1      \\
				\vdots     &  \ddots     &    \ddots    & \ddots     & \vdots     & \vdots   \\
				1          &    1        &    \ddots    & r   s^{-1} &  s         &   1      \\
				1          &    1        &    \cdots    &  r^{-1}    & r   s^{-1} &   s^2    \\
				1          &    1        &    \cdots    &    1       &   r^{-2}   & r^2 s^{-2}
			\end{array}\right),\hspace{0.1cm}
			R+S  =
			\left(\begin{array}{cccccc}
				&   1      &              &         &         &          \\
				-1  &          &  \ddots      &         &         &          \\
				& \ddots   &              & \ddots  &         &          \\
				&          &    \ddots    &         &    1    &          \\
				&          &              &    -1   &         &   2      \\
				&          &              &         &   -2    &
			\end{array}\right). \\
			$}
	\end{equation*}

	For type $D_n$, we have
	$$\operatorname{ det } (R+S) = \begin{cases}
		4,       & 2 \mid  n, \\
		0,       & 2 \nmid n,
	\end{cases} $$
	where
	\begin{equation*}\resizebox{1\hsize}{!}{$
			A =
			\left(\begin{array}{cccccc}
				r   s^{-1} &    s        &    \cdots    &    1       &    1       &   1      \\
				r^{-1}     &   r s^{-1}  &    \ddots    &    1       &    1       &   1      \\
				\vdots     &  \ddots     &    \ddots    & \ddots     & \vdots     & \vdots   \\
				1          &    1        &    \ddots    & r   s^{-1} &  s         &   s      \\
				1          &    1        &    \cdots    &  r^{-1}    &    rs^{-1} &   rs    \\
				1          &    1        &    \cdots    &  r^{-1}    &r^{-1}s^{-1}& r s^{-1}
			\end{array}\right), \hspace{0.1cm}
			R+S =
			\left(\begin{array}{cccccc}
				&   1      &              &         &         &          \\
				-1  &          &   \ddots     &         &         &          \\
				&  \ddots  &              & \ddots  &         &          \\
				&          &    \ddots    &         &    1    &    1     \\
				&          &              &    -1   &         &   2      \\
				&          &              &    -1   &   -2    &
			\end{array}\right).
			$}
	\end{equation*}
	
	For the family of exceptional types $E_6, E_7, E_8$, it suffices to list the data of type $E_8$ (from which we can read off the data of type $E_6, E_7$).
	\begin{align*}
		& A =
		\left(\begin{array}{cccccccc}
			rs^{-1} &    s     &    1     &    1    &    1    &    1    &    1    &   1      \\
			r^{-1}  & rs^{-1}  &    s     &    1    &    1    &    1    &    1    &   1      \\
			1       & r^{-1}   & rs^{-1}  &    s    &    1    &    1    &    1    &   1      \\
			1       &    1     &  r^{-1}  & rs^{-1} &    s    &    1    &    1    &   1      \\
			1       &    1     &    1     & r^{-1}  & rs^{-1} &    s    &    s    &   1      \\
			1       &    1     &    1     &    1    &  r^{-1} & rs^{-1} &    1    &   1      \\
			1       &    1     &    1     &    1    &  r^{-1} &    1    & rs^{-1} &   s      \\
			1       &    1     &    1     &    1    &    1    &    1    &  r^{-1} & rs^{-1}
		\end{array}\right), \\
		& R+S =
		\left(\begin{array}{cccccccc}
			&    1     &          &         &         &         &         &          \\
			-1      &          &    1     &         &         &         &         &          \\
			&   -1     &          &    1    &         &         &         &          \\
			&          &    -1    &         &    1    &         &         &          \\
			&          &          &   -1    &         &    1    &    1    &          \\
			&          &          &         &   -1    &         &         &          \\
			&          &          &         &   -1    &         &         &    1     \\
			&          &          &         &         &         &   -1    &
		\end{array}\right).
	\end{align*}
	Hence, for types $E_6$ and $E_8$, we have $\operatorname{ det } (R+S) = 1$; for type $E_7$, $\operatorname{ det } (R+S) = 0.$
	
	For type $F_4$, we have
	$\operatorname{ det } (R+S) = 4$,
	where
	\begin{align*}
		&A =
		\left(\begin{array}{cccc}
			r^{2} s^{-2} &    s^2        &     1      &    1          \\
			r^{-2}        & r^{2} s^{-2} &   s^{2}    &    1          \\
			1             &  r^{-2}       & r   s^{-1} &    s          \\
			1             &  1            &   r^{-1}   &    rs^{-1}
		\end{array}\right),
		& R+S =
		\left(\begin{array}{cccc}
			&     2    &          &                       \\
			- 2     &          &     2    &                       \\
			&     -2   &          &     1                 \\
			&          &   -1     &
		\end{array}\right).
	\end{align*}
	
	For type $G_2$, we have
	$\operatorname{ det } (R+S) = 9$, where
	\begin{align*}
		& A =
		\left(\begin{array}{cc}
			r^3 s^{-3}   &    s^3      \\
			r^{-3}     & r s^{-1}
		\end{array}\right),
		& R+S =
		\left(\begin{array}{cc}
			&   3       \\
			-3      &
		\end{array}\right).
	\end{align*}
	When $n$ is odd, we have $\operatorname{corank} \; (R+S) = 1$. We list the unique non-zero solution (up to scalar) for each type as follows
	
	\begin{center}
		\begin{tabular}{|c|l|}
			\hline
			$A_{2k+1}: \; \mathfrak{sl}_{2k+2}$&
			$v^{*}=(1,0,1,0,\cdots,1,0,1)^T$\\
			$B_{2k+1}:\; \mathfrak{so}_{4k+3}$&
			$v^{*}=(1,0,1,0,\cdots,1,0,1)^T$\\						
			$C_{2k+1}:\; \mathfrak{sp}_{4k+2}$&
			$v^{*}=(2,0,2,0,\cdots,2,0,1)^T$\\	
			$D_{2k+1}:\; \mathfrak{so}_{4k+2}$&
			$v^{*}=(2,0,2,0,\cdots,2,1,-1)^T$\\		
			$E_7$&
			$v^{*}=(1,0,1,0,1,0,0)^T$\\		
			\hline									
		\end{tabular}
	\end{center}				
	
	This completes the proof.
\end{proof}

\subsection{Weight modules and the category $\mathcal{O}^{r,s}_f$}

Recall the structure of weight modules studied by \cite{BGH07,PHR10}.

Let $\varrho: U^0 \longrightarrow \mathbb{K}$ be an algebraic homomorphism and $V^\varrho$ be a $1$-dimensional $\mathcal{B}$-module. Denote $M(\varrho) = U \otimes_\mathcal{B} V^\varrho$ as the Verma module with the highest weight $\varrho$ and $L(\varrho)$ as its unique irreducible quotient.

Let $\lambda \in \Lambda$ and rewrite it as $\lambda = \sum_{i=1}^{n} \lambda_i \alpha_i,\; \lambda_i \in \mathbb{Q}.$ Then we can define an algebraic homomorphism $\varrho^\lambda : U^0 \rightarrow \mathbb{K},$ satisfying
\begin{align*}
	&\varrho^{\lambda} (\omega_j) := \prod\limits_{i=1}^n \langle \omega'_i,\omega_j \rangle^{\lambda_i}
	= \prod\limits_{i=1}^n a_{ji}^{\lambda_i}.
\end{align*}
Clearly, it satisfies the property $\varrho^{\lambda+\mu} = \varrho^\lambda \varrho^\mu.\,$

When $\lambda \in Q$, one would get following relations from the Hopf pairing.
\begin{align*}
	\varrho^\lambda(\omega_j) = \langle \omega'_\lambda\ , \omega_j \rangle,\qquad \varrho^\lambda(\omega'_j) = \langle \omega'_j , \omega_{-\lambda}  \rangle.
\end{align*}
For convenience, we denote $M(\lambda) := M(\varrho^\lambda)$ and $ L(\lambda):= L(\varrho^\lambda)$ when $\lambda \in \Lambda$.
\begin{lemma}
	\textup{(cf., \cite{BKL06,BGH07,PHR10})}   \label{irrmod}
	For the two-parameter quantum group $U_{r,s}(\mathfrak{g})$, we have
	
	$(1)$ Let $v_\lambda$  be a highest weight vector of $M(\lambda)$ for $\lambda \in \Lambda^+$. Then
	\[ L(\lambda) = M(\lambda) \big/ (\sum\limits_{i=1}^{n} U f_i^{(\lambda,\alpha_i^\vee) +1} \cdot v_\lambda ).\]
	is a finte dimensional irreducible $U_{r,s}(\mathfrak{g})$-module. Also, it has the decomposition of weight space  $L(\lambda) = \bigoplus_{\eta \leqslant \lambda} L(\lambda)_\eta$, where
	\[L(\lambda)_\eta = \{ x \in L(\lambda) \mid \omega_i . x = \varrho^\eta(\omega_i) x,\  \omega'_i . x = \varrho^\eta(\omega'_{i}) x,\ 1 \leqslant i \leqslant n \}.\]
	
	$(2)$ The elements $e_i ,f_i \ (1\leqslant i \leqslant n)$ act locally nilpotently on $L(\lambda)$.
\end{lemma}

\begin{theorem}\textup{(cf., \cite{BKL06,BGH07,PHR10})}  \label{Weyl-dim}
	Suppose $rs^{-1}$ is not a root of unity, when $\lambda \in \Lambda^+$, we have
	\[dim \; L(\lambda)_\eta = dim \; L(\lambda)_{\sigma(\eta)}, \quad \forall\; \eta \in \Lambda, \; \sigma \in W.\]
\end{theorem}

\begin{defi} \textup{(cf., \cite{HP12})}
	The category $\mathcal{O}^{r,s}_f$ consists of finite-dimensional $U_{r,s}(\mathfrak{g})$-modules $V$ $($of type $1$$)$ satisfying the following conditions:
	
	$(1)$ $V$ has a weight space decomposition $V = \bigoplus_{\lambda\in\Lambda} V_\lambda,$
	\begin{align*}
		V_\lambda = \{ v \in V \mid \omega_\eta v = r^{\langle\lambda,\eta\rangle} s^{-\langle \eta,\lambda\rangle} v,\ \omega'_\eta v = r^{-\langle \eta,\lambda\rangle} s^{\langle\lambda,\eta\rangle} v,\quad \forall\; \eta \in Q\;\},
	\end{align*} where $\operatorname{dim} (V_\lambda)$ is finite for all $\lambda \in \Lambda;$
	
	$(2)$ there exist a finite number of weights $\lambda_1,\cdots,\lambda_t \in \Lambda$, such that
	\begin{align*}
		\operatorname{Wt}(V) \subseteq \bigcup_{i=1}^t  D(\lambda_i),
	\end{align*}		
	where $D(\lambda):=\{\mu \in \Lambda \mid \mu < \lambda\}.$ The morphisms are $U_{r,s}(\mathfrak{g})$-module homomorphisms.
\end{defi}

\subsection{Rosso form and characters}

\begin{defi} \textup{(cf., \cite{BKL06,BGH06})}  \label{Rosso}
	The bilinear form $\langle - \; | \;- \rangle : U \times U \longrightarrow \mathbb{K}$ defined by
	\begin{equation*} \resizebox{1\hsize}{!}{$
			\begin{split}
				\langle (y_1 {\omega'_{\nu_1}}^{-1}) \omega_{\eta_1}^{\prime} \omega_{\phi_1} x_1\;\mid\; (y_2 {\omega'_{\nu_2}}^{-1}) \omega_{\eta_2}^{\prime} \omega_{\phi_2} x_2 \rangle &=
				\left\langle y_2, x_1 \right\rangle
				\left\langle\omega_{\eta_2}^{\prime}, \omega_{\phi_1}\right\rangle
				\left\langle\omega_{\eta_1}^{\prime}, \omega_{\phi_2}\right\rangle
				\left\langle S^{2}\left(y_1 \right), x_2\right\rangle \\
				&= (rs^{-1})^{(\rho,\nu_1)}	\left\langle y_2, x_1 \right\rangle
				\left\langle\omega_{\eta_2}^{\prime}, \omega_{\phi_1}\right\rangle
				\left\langle\omega_{\eta_1}^{\prime}, \omega_{\phi_2}\right\rangle
				\left\langle y_1, x_2\right\rangle
			\end{split}
			$}
	\end{equation*}
	is called the Rosso form of the two-parameter quantum group $U=U_{r,s}(\mathfrak{g})$, where $\rho$ is the half sum of positive roots, and $x_i \in U^+_{\mu_i},\; y_i \in U^-_{-\nu_i},\; \mu_i, \nu_i \in Q^+$.
\end{defi}
\begin{theorem} \textup{(cf., \cite{BKL06,BGH06,PHR10})} \label{inv}
	The Rosso form $\langle -\;|\;- \rangle$ is $\operatorname{ad}_l$-invariant, i.e.,
	\begin{align*}
		\langle  \operatorname{ad}_l (a) b \;|\; c \rangle  = \langle b \;|\; \operatorname{ad}_l (S(a)) c \rangle, \quad \forall\; a,b,c \in U.
	\end{align*}
	
\end{theorem}

\begin{theorem} \textup{(cf., \cite{BKL06,BGH07,PHR10})} \label{nondegen for grading}
	For any $\beta \in Q^+,$ the restriction of the skew-pairing $\langle -,- \rangle$ to $\mathcal{B}'_{-\beta} \times \mathcal{B}_\beta$ is nondegenerate.
\end{theorem}

\begin{lemma} \label{orth}
	If $\mu_i, \nu_i \in Q^+$, then $\langle U^-_{-\nu_1} U^0 U^+_{\mu_1} \mid U^-_{-\nu_2} U^0 U^+_{\mu_2} \rangle \neq 0$ if and only if $\mu_1 = \nu_2$, $\nu_1 = \mu_2$.
\end{lemma}

\begin{defi} \label{nondegen}
	Define a group homomorphism $\chi_{\eta,\phi}: Q \times Q \longrightarrow \mathbb{K}^\times$ by
	\begin{align*}
		\chi_{\eta,\phi} (\eta',\phi') = \langle \omega'_\eta, \omega_{\phi'}  \rangle\langle \omega'_{\eta'},\omega_{\phi} \rangle,
	\end{align*}
	where $(\eta,\phi),(\eta',\phi') \in Q \times Q, \; \mathbb{K}^\times = \mathbb{K} \setminus \{ 0 \} .$
\end{defi}
\begin{lemma} \label{chi}
	If $\chi_{\eta, \phi} = \chi_{\eta', \phi'}$, then  $(\eta,\phi) = (\eta',\phi').$
\end{lemma}
\begin{proof} Let $\eta  = \sum_{i=1}^{n} \eta_i \alpha_i,\; \eta'  = \sum_{i=1}^{n} \eta'_i \alpha_i,$ and $ \phi,\phi' \in Q$, for  $j =1,\cdots,n$ we have
	\begin{align*}
		\chi_{\eta,\phi} (0,\alpha_j)
		= \langle \omega'_\eta , \omega_j \rangle = \prod_{i=1}^{n} \langle \omega'_i ,\omega_j \rangle^{\eta_i}, \hspace{1cm}
		&\chi_{\eta',\phi'}(0,\alpha_j)
		= \prod_{i=1}^{n}  \langle \omega'_i ,\omega_j \rangle^{\eta'_i}.
	\end{align*}	
	Since $\chi_{\eta, \phi} = \chi_{\eta', \phi'}$, we have
	\begin{align*}	
		1=\frac{\chi_{\eta,\phi}(0,\alpha_j)}{\chi_{\eta',\phi'}(0,\alpha_j)}
		&= \prod_{i=1}^{n}  \langle \omega'_i ,\omega_j \rangle^{\eta_i - \eta'_i} = \prod_{i=1}^{n} a_{ji}^{\eta_i - \eta'_i} \\
		&=\prod_{i=1}^{n} r^{R_{ji}(\eta_i - \eta'_i)} s ^{S_{ji}(\eta_i - \eta'_i)},
	\end{align*}
	which in turn gives $(R-S) (\eta - \eta') = 0$.  Finally, Proposition \ref{sym} yields $\eta = \eta'$. A similar argument leads to the conclusion $\phi = \phi'$.
\end{proof}

\begin{theorem}
	The Rosso form $\langle - \;|\; - \rangle$ of $U_{r,s}(\mathfrak{g})$ is nondegenerate.
\end{theorem}
\begin{proof}
	Since the skew-pairing $ \langle - , - \rangle$ has orthogonality for the grading, it suffices to check the case when $u \in U^-_{-\nu} U^0 U^+_{\mu},$ if $\langle u \;|\; v  \rangle = 0$ holds for all $ v \in U^-_{-\mu} U^0 U^+_{\nu}$, then $u=0$.
	
	Denote $d_\mu = \operatorname{dim}\; U^+_\mu$. Let $\{ u^\mu_1,\cdots ,u^\mu_{d_\mu}\}$ be a basis of $U^+_\mu$ and  $\{ v^\mu_1,\cdots,v^\mu_{d_\mu} \}$ be its dual basis in $U^-_{-\mu}$ with respect to the skew-Hopf pair, that is, $\langle v^\mu_i , u^\mu_j \rangle = \delta_{ij}$.
	Hence $U^-_{-\nu} U^0 U^+_\mu = \operatorname{span}_\mathbb{K}\{ (v^\nu_i {\omega'_{\nu}}^{-1}) \omega'_\eta \omega_\phi u^\mu_j \mid 1\leqslant i \leqslant d_\nu,\; 1\leqslant j \leqslant d_\mu\}.$ Notice that
	\begin{align*}
		\langle  (v^\nu_i {\omega'_{\nu}}^{-1}) \omega'_\eta \omega_\phi u^\mu_j \;|\;  (v^\mu_k {\omega'_\mu}^{-1}) \omega'_{\eta'} \omega_{\phi'} u^\nu_l \rangle &= \langle \omega'_\eta , \omega_{\phi'}  \rangle \langle \omega'_{\eta'}, \omega_\phi \rangle
		\langle v^\mu_k , u^\mu_j \rangle \langle S^2(v^\nu_i),u^\nu_l \rangle \\
		&= \delta_{kj}\delta_{il} (rs^{-1})^{(\rho,\nu)}  \langle \omega'_\eta , \omega_{\phi'}  \rangle \langle \omega'_{\eta'}, \omega_\phi \rangle.
	\end{align*}
	Let $u=\sum_{i,j,\eta,\phi} k_{i,j,\eta,\phi} (v^\nu_i {\omega'_\nu}^{-1})  \omega'_\eta \omega_\phi u^\mu_j,\; v = (v^\mu_k  {\omega'_\mu}^{-1})\omega'_{\eta'} \omega_{\phi'} u^\nu_l,$ where $1\leqslant k \leqslant d_\mu,\; 1 \leqslant l \leqslant d_\nu,\; \eta',\phi' \in Q.$ Suppose $\langle u \;|\; v \rangle=0$,  we have
	\begin{align*}
		0 &= \sum_{\eta,\phi} k_{k,l,\eta,\phi} (rs^{-1})^{(\rho,\nu)}  \langle \omega'_\eta , \omega_{\phi'}  \rangle \langle \omega'_{\eta'}, \omega_\phi \rangle\\
		&= \sum_{\eta,\phi} k_{k,l,\eta,\phi} (rs^{-1})^{(\rho,\nu)} \chi_{\eta,\phi}(\eta',\phi'),
	\end{align*}
	By Dedekind theorem, we have $k_{l,k,\eta,\phi} = 0$. so $u=0.$
\end{proof}

\subsection{Harish-Chandra homomorphism}

Let $Z(U)$ be the centre of $U_{r,s}(\mathfrak{g})$. It follows that $Z(U) \subseteq U_0$. Now we define an algebra homomorphism $\gamma^{-\rho} : U^0 \longrightarrow U^0$ as
\[ \gamma ^{-\rho}(\omega'_\eta \omega_\phi) = \varrho^{-\rho} (\omega'_\eta \omega_\phi) \omega'_\eta \omega_\phi. \]
Particularly, we have
\[ \gamma^{-\rho} (\omega'_i \omega^{-1}_i) = (r_i s_i^{-1})^{(\rho,\alpha_i^\vee)} \omega'_i \omega^{-1}_i = (r_i s_i^{-1})\,\omega'_i \omega^{-1}_i. \]

\begin{defi}
	Denote $\xi : Z(U) \longrightarrow U^0$ as the restricted map $\gamma^{-\rho}\circ \pi |_{Z(U)},$
	\[\gamma^{-\rho} \pi : U_0 \longrightarrow U^0 \longrightarrow U^0,\]
	where $\pi : U_0 \rightarrow U^0$ is the canonical projection. We call $\xi$ the  Harish-Chandra homomorphism of $U$.
\end{defi}

Define a subalgebra $U^0_\flat= \bigoplus\limits_{\eta \in Q} \mathbb{K} \omega'_\eta \omega_{-\eta}  $ and let the Weyl group $W$ act on it
\[\sigma(\omega'_\eta \omega_{-\eta}) = \omega'_{\sigma(\eta)} \omega_{-\sigma(\eta)},\, \forall\; \sigma \in W,\; \eta \in Q.\,\]

\begin{theorem}
	\cite{BKL06,BGH06,G10} For types $A_n, B_n$ and $G_2$, when rank $n$ is even, the Harish-Chandra homomorphism $\xi : Z(U) \longrightarrow (U_\flat^0)^W$ is an algebra isomorphism.
\end{theorem}

\section{Harish-Chandra homomorphism $\xi$ is injective}

For each $i$, there are two skew-derivations $\partial_i, {}_i \partial: U^-_{-\zeta} \rightarrow U^-_{-\zeta+\alpha_i},\; \zeta \in Q^+$ \cite{BGH07}, defined by
\begin{equation*}
	\begin{split}
		\partial_i (1) &= 0, \quad
		\partial_i (f_j) = \delta_{ij}, \quad
		\partial_i (y y') =  \partial_i (y) y' + \langle \omega'_\zeta, \omega_i \rangle y \partial_i (y'),\\
		{}_i \partial (1) &= 0, \quad
		{}_i \partial (f_j) = \delta_{ij}, \quad
		{}_i \partial (y y') = \langle \omega'_i, \omega_{\zeta'} \rangle {}_i \partial (y) y' + y \hspace{0.05cm} {}_i \partial (y'),\\
	\end{split}
\end{equation*}
for all $y \in U^-_{-\zeta},\; y' \in U^-_{-\zeta'}$ and $x \in U^+$, which have the following properties

\begin{lemma} \textup{(cf., \cite{BGH07})} \label{partial}
	\textnormal{(1)} $\langle y, e_i x \rangle = \langle f_i,e_i \rangle \langle \partial_i(y), x \rangle = (s_i - r_i)^{-1} \langle \partial_i(y) , x \rangle$,
	
	\textnormal{(2)} $\langle y, x e_i \rangle = \langle f_i,e_i \rangle \langle {}_i \partial (y), x \rangle = (s_i - r_i)^{-1} \langle {}_i \partial (y) , x \rangle$,	
	
	\textnormal{(3)} $e_i y -y e_i = (r_i - s_i)^{-1} (\omega_i \partial_i(y)  -  {}_i \partial(y) \omega'_i)$,
	
	\textnormal{(4)} If for all $i$, we have $\partial_i (y) = 0$, then $y$ = 0,
	
	\textnormal{(5)} If for all $i$, we have ${}_i \partial (y) = 0$, then $y$ = 0.
\end{lemma}

\begin{theorem}\label{inj if}
	The Harish-Chandra homomorphism $\xi : Z(U) \longrightarrow U^0$ is injective.
\end{theorem}
\begin{proof}
	Consider the triangular decomposition of $U = U^- U^0 U^+$ and set $K = \bigoplus_{\nu>0} U^-_{-\nu} U^0 U^+_{\nu}$, which is a two-sided ideal of $U_0 = U^0 \oplus K$ and $ K= ker(\pi)$.
	The following argument shows that for $z\in Z(U)$ with $\xi (z) = 0$, we will get $z = 0$.
	
	Let $z=\sum_{\nu \in Q^+} z_\nu$, where $ z_\nu \in U^-_{-\nu} U^0 U^+_{\nu}$. The condition $\xi (z) = 0$ implies $z_0=0$. Let us assume that $z\ne0$, then there exists a minimal $\nu \in Q^+$ such that $z_\nu \neq 0,\; \nu>0$ , and choose bases $\{x_l\}$ and $\{y_k\}$ for spaces
	$U^+_{+\nu}$ and $ U^-_{-\nu}$, respectively. Now we write $z_\nu = \sum_{k,l} y_k t_{k,l} x_l, \ t_{k,l} \in U^0$.  Then for all $i=1,\cdots,n$, we have
	\begin{align*}
		0 &=e_{i} z-z e_{i} \\
		&=\sum_{\gamma \neq \nu}\left(e_{i} z_{\gamma}-z_{\gamma} e_{i}\right)+ \left(e_{i} z_{\nu}-z_{\nu} e_{i}\right) \\
		&=\sum_{\gamma \neq \nu}\left(e_{i} z_{\gamma}-z_{\gamma} e_{i}\right)+\sum_{k, l}\left(e_{i} y_{k}-y_{k} e_{i}\right) t_{k, l} x_{l}+\sum_{k, l} y_{k}\left(e_{i} t_{k, l} x_{l}-t_{k, l} x_{l} e_{i}\right).
	\end{align*}
	Since $e_i y_k - y_k e_i = (r_i - s_i)^{-1} (\omega_i \partial_i(y)  -  {}_i \partial(y) \omega'_i) \in U^-_{-(\nu-\alpha_i)} U^0$, only the second term in the equation falls in $U^-_{-(\nu-\alpha_i)}U^0 U^+_{\nu}$, so it is forced to be $0,$ that is
	\begin{align*}
		\sum_{k, l}\left(e_{i} y_{k}-y_{k} e_{i}\right) t_{k, l} x_{l} = 0.
	\end{align*}
	Now write $t_{k,l} = \sum_{\eta,\phi \in Q} g^{k,l}_{\eta,\phi} \omega'_\eta \omega_\phi$, then for all $k,l,\eta,\phi$, if there exists a maximal $\eta_0$ and some $k_0, l_0, \phi_0$ such that $g^{k_0,l_0}_{\eta_0,\phi_0}\neq 0$ then we have
	\begin{align*}
		0 &=\sum_{k, l}\left(e_{i} y_{k}-y_{k} e_{i}\right) t_{k, l} x_{l} \\
		& = (r_i - s_i)^{-1} \sum_{k, l, \eta,\phi}  g^{k,l}_{\eta,\phi} (\omega_i \partial_i(y_k) - {}_i \partial (y_k) \omega'_i) \omega'_\eta \omega_\phi x_l
	\end{align*}
	for all $i$. Since $\eta_0$ is maximal, then $g^{k_0,l_0}_{\eta_0,\phi_0} {}_i \partial (y_{k_0}) \omega'_{\eta_0+\alpha_i} \omega_{\phi_0} x_{l_0}$ has to be zero for all $i$, which force ${}_i \partial (y_{k_0}) = 0$ for all $i$. Then by Lemma \ref{partial} (5), we have $y_{k_0} = 0$, which is a contradiction. Hence all $t_{k,l} = 0$, and then $z_{\nu}=0$. Finally we have $z=0$.
\end{proof}

\begin{remark}
	Note that in the case when rank $n$ is odd, the proof of proposition 5.2 $($see p. 458, -line 3 \cite{BKL06}$)$ for type $A_n$ really contains a gap, as was pointed out in Remark 3.5 \cite{HS14}.
\end{remark}

\section{The image of the Harish-Chandra homomorphism $\xi$ with even rank}

Define an algebra homomorphisms $\varrho^{\lambda,\mu}$ from $U^0$ to $\mathbb{K}$ as $\varrho^{\lambda,\mu}= \varrho^{0,\lambda} \varrho^{\mu,0}, \lambda,\mu$ $ \in$ $ \Lambda$, where
\begin{align*}
	&\varrho^{0,\mu}: 	\omega'_\eta \omega_\phi \mapsto (rs^{-1})^{(\eta+\phi,\mu)}, \\
	&\varrho^{\lambda,0}:
	\omega'_\eta \omega_\phi \mapsto \varrho^\lambda(\omega'_\eta \omega_\phi).
\end{align*}

\begin{lemma} \label{nondegen of varho}
	Suppose $n$ is even, then
	
	$(1)$ let $u = \omega'_\eta \omega_\phi,\; \eta,\phi \in Q. \,$ If $\varrho^{\lambda,\mu} (u) = 1$ for all $\lambda,\mu \in \Lambda$, then $u=1;$
	
	$(2)$ if $u \in U^0$ satisfying $\varrho^{\lambda,\mu} (u) = 0$ for all $\lambda,\mu \in \Lambda$, then $u=0.$
\end{lemma}
\begin{proof}
	(1) Fix $\eta,\ \phi \in Q, \; \lambda \in \Lambda$,
	and write them as $\eta = \sum\limits_{i=1}^n \eta_i \alpha_i, \ \phi=\sum\limits_{i=1}^n \phi_i \alpha_i,\ \lambda = \sum\limits_{i=1}^n \lambda_i \alpha_i,$ $   \eta_i,\ \phi_i \in \mathbb{Z},\ \lambda_i \in \mathbb{Q}$.
	For simplicity, we use the same symbols to represent the column vectors of
	$\eta, \phi$ and $\lambda$ respect to the basis $\{\alpha_i\}_{i=1}^n$.
	By $1 = \varrho^{\lambda,\mu} (u) = \varrho^\lambda(u) \varrho^{0,\mu} (u)$ and
	\begin{align*}
		\varrho^\lambda (\omega_\phi) 		
		& = \prod_{j=1}^{n} (\varrho^\lambda(\omega_j))^{\phi_j}
		=  \prod_{j,i=1}^{n} a_{ji}^{\lambda_i \phi_j}  \label{rho^lambda}= \prod_{j,i=1}^{n} r^{R_{ji} \lambda_i \phi_j} s^{S_{ji} \lambda_i \phi_j}  \\
		& 			
		= r^{\phi^T R \lambda} \cdot s^{\phi^T S \lambda}
		=r^{\lambda^T R^T \phi} \cdot s^{\lambda^T S^T\phi },\\
		\varrho^\lambda (\omega'_\eta)
		& = s^{\eta^T R \lambda}\cdot r^{\eta^T S \lambda}=s^{\lambda^T R^T \eta}\cdot r^{\lambda^T S^T\eta},
	\end{align*}
	we have
	\begin{gather*}
		\varrho^\lambda (\omega'_\eta \omega_\phi)
		= r^{\lambda^T (R^T \phi + S^T \eta)} s^{\lambda^T(S^T \phi + R^T \eta)}, \\
		\varrho^{0,\mu} (\omega'_\eta \omega_\phi) =(rs^{-1})^{(\eta+\phi,\mu)}.
	\end{gather*}
	It follows that
	\[
	\begin{cases}
		\hspace{0.3cm} (\eta + \phi,\mu) + \lambda^T (R^T \phi + S^T \eta) = 0, \\
		-(\eta + \phi,\mu) + \lambda^T (S^T \phi + R^T \eta) = 0.	
	\end{cases}
	\]
	Set $\lambda = 0$, we have $\eta+\phi = 0.\,$ The invertible $R-S$ leads to $\eta = \phi = 0$, i.e., $u =1.$
	
	(2) Fixing a pair $(\eta, \phi) \in Q\times Q,$ one can define a character $\kappa_{\eta,\phi} $ on the group $\Lambda \times \Lambda$ to be $\kappa_{\eta,\phi}:  (\lambda,\mu) \mapsto \varrho^{\lambda,\mu} (\omega'_\eta \omega_\phi).$
	
	Let $u = \sum\limits_{(\eta,\phi)} k_{\eta,\phi}  \omega'_\eta \omega_\phi, \; k_{\eta,\phi} \in \mathbb{K}.\,$ Then
	\[ 0 = \varrho^{\lambda,\mu} (u)
	= \sum\limits_{(\eta,\phi)} k_{\eta,\phi} \varrho ^{\lambda,\mu} (\omega'_\eta \omega_\phi)
	=  \sum\limits_{(\eta,\phi)} k_{\eta,\phi} \kappa_{\eta,\phi} (\lambda,\mu).
	\]
	Since characters $\{\kappa_{\eta,\phi}\}$ are different from each other, we have $k_{\eta,\phi} = 0$, that is, $u = 0.$
\end{proof}

\begin{proposition} \label{U0bWeyl}
	$\varrho^{\sigma(\lambda),\mu}(u) = \varrho^{\lambda,\mu} (\sigma^{-1}(u))$, for $u \in U^0_\flat,\; \sigma \in W,\; \lambda,\; \mu \in \Lambda$.
\end{proposition}
\begin{proof}
	It is sufficient to check the case that $u = \omega'_\eta \omega_{-\eta}$ and $\sigma = \sigma_i$ a simple reflection.
	As we did earlier, use subscript $\alpha$ to denote ones column vectors with respect to the basis $\{\alpha_i\}$. By Corollary \ref{InvR-S}, we have
	\begin{align*}
		\varrho^{\lambda,0}(\sigma_i^{-1} (u))
		&=\varrho^\lambda (\omega'_{\sigma_i (\eta)} \omega_{-\sigma_i (\eta)} ) \\
		&= r^{ \lambda^T (R^T(\sigma_i (\eta)) + S^T(\sigma_i(-\eta)) )  }
		s^{ \lambda^T (S^T(\sigma_i (\eta)) + R^T(\sigma_i(-\eta)) )  } \\
		&= r^{ \lambda^T (R-S) \Sigma_i \cdot \eta}
		s^{ \lambda^T (S-R) \Sigma_i \cdot \eta}\\
		&= r^{ \lambda^T \Sigma_i^T \cdot(R-S)  \eta}
		s^{ \lambda^T \Sigma_i^T \cdot(S-R)  \eta}\\
		&= r^{ (\sigma_i(\lambda))^T (R^T \eta + S^T(-\eta))}
		s^{ (\sigma_i(\lambda))^T (S^T \eta + R^T(-\eta))} \\
		&=\varrho^{\sigma_i(\lambda)} (\omega'_\eta \omega_{-\eta})
		=\varrho^{\sigma_i(\lambda),0} (u), \\
	\end{align*}
	\begin{align*}
		\varrho^{0,\mu}(\sigma^{-1} (u))
		&= \varrho^{0,\lambda} (\sigma^{-1} (\omega'_\eta \omega_{-\eta})) \\
		&= (rs^{-1})^{( \sigma^{-1} (\eta) + \sigma^{-1}(-\eta),\mu)}\\
		&= (rs^{-1})^{(0,\mu)} = \varrho^{0,\mu} (\omega'_\eta \omega_{-\eta}).
	\end{align*}
	So $\varrho^{\lambda,\mu}(\sigma^{-1}(u)) = \varrho^{\sigma(\lambda), \mu} (u)$.
\end{proof}
Define a subalgebra $(U^0_\flat)^{W}=\{u \in U^0_\flat \;|\; \sigma(u)=u,\, \forall  \sigma \in W\}$ and characters $\kappa_{\eta,\phi}: (\lambda,\mu) \mapsto \varrho^{\lambda,\mu}(\omega'_\eta \omega_\phi),$ on $\Lambda \times \Lambda$ for each $(\eta,\phi) \in Q \times Q$. Further we define $\kappa^i_{\eta,\phi}: (\lambda,\mu)\mapsto \varrho^{\sigma_i(\lambda),\mu}(\omega'_\eta \omega_\phi).$

\begin{lemma} \label{If in}  Suppose that rank $n$ is even, $ \sigma \in W,\; \lambda, \mu \in \Lambda.$ If $u\in U^0$ satisfies that $\varrho^{\sigma(\lambda),\mu}(u) = \varrho^{\lambda,\mu} (u)$, then $u\in (U^0_\flat)^W.$
\end{lemma}
\begin{proof}
	Let $u = \sum\limits_{(\eta,\phi)} k_{\eta,\phi} \omega'_\eta \omega_\phi \in U^0\,$, then
	\begin{align*}
		\sum\limits_{(\eta,\phi)} k_{\eta,\phi} \varrho^{\lambda,\mu} (\omega'_\eta \omega_\phi) =
		\sum\limits_{(\zeta,\psi)} k_{\zeta,\psi} \varrho^{\sigma(\lambda),\mu} (\omega'_\zeta \omega_\psi),	
	\end{align*}
	hence we have an equation for characters:
	\begin{align*}
		\sum\limits_{(\eta,\phi)} k_{\eta,\phi} \kappa_{\eta,\phi} =
		\sum\limits_{(\zeta,\psi)} k_{\zeta,\psi} \kappa^i_{\zeta,\psi}.
	\end{align*}
	Comparing the two sides of the equation, for each $k_{\eta,\phi} \neq 0$, there exists one $(\zeta,\psi) \in Q \times Q,$ such that $ \kappa_{\eta,\phi} = \kappa^i_{\zeta,\psi} $ and $ k_{\zeta,\psi} = k_{\eta,\phi}.\,$ Then
	\begin{equation*}
		\begin{split}
			\kappa_{\eta,\phi}(0,\varpi_j) &
			= \varrho^{0,\varpi_j} (\omega'_\eta \omega_\phi) = (rs^{-1})^{(\eta+\phi,\varpi_j)}\\
			&=\kappa_{\zeta,\phi}^i (0,\varpi_j) \\
			&= \varrho^{0,\varpi_j}(\omega'_\zeta \omega_\psi) = (rs^{-1})^{(\zeta+\phi,\varpi_j)}
		\end{split}
	\end{equation*}
	yields $\eta+\phi = \zeta+\psi$, and
	\begin{align*}
		\varrho^{\varpi_i} (\omega'_\eta \omega_\phi)
		= \varrho^{\sigma_i(\varpi_i)} (\omega'_\zeta \omega_\phi)
		= \varrho^{\varpi_i - \alpha_i}(\omega'_\zeta \omega_\phi).
	\end{align*}
	Since $\phi = \zeta+\psi-\eta$, it follows that
	\begin{align*}
		\varrho^{\varpi_i}(\omega'_{\eta-\zeta} \omega_{-(\eta-\zeta)})
		\varrho^{\alpha_i}(\omega'_\zeta \omega_{-\zeta}) =
		\varrho^{-\alpha_i}(\omega_{\zeta+\phi}).
	\end{align*}
	Rewrite the equation in the form of $r^k s^l = 1$, for $i= 1,\cdots,n$, we have
	\begin{align*}
		(\varpi_i) ^T (R-S)^T (\eta-\zeta) + \varepsilon_i^T (R-S)^T \zeta + \varepsilon_i^T S^T (\zeta+\psi) = 0, \\
		(\varpi_i) ^T (S-R)^T (\eta-\zeta) + \varepsilon_i^T (S-R)^T \zeta + \varepsilon_i^T R^T (\zeta+\psi) = 0,
	\end{align*}
	where $\varepsilon_i$ is the $i$-th unit vector. Adding two equations, one has $\varepsilon_i^T (R+S)^T (\zeta+\psi) = 0,\;i= 1,\cdots,n.\,$ That is, $(R+S)^T (\zeta+\psi) = 0.$
	Since $n$ is even, det$(R+S) \neq 0,$ we have $\zeta+\psi=0=\phi+\eta.\,$ Then
	\begin{align*}
		u = \sum\limits_{(\eta,-\eta)} k_{\eta,-\eta} \omega'_\eta \omega_{-\eta} \in U^0_\flat.
	\end{align*}
	Finally, by Lemma \ref{U0bWeyl},
	\begin{align*}
		\varrho^{\lambda,\mu} (\sigma^{-1}(u)) = \varrho^{\sigma(\lambda),\mu} (u)  = \varrho^{\lambda,\mu}(u),\quad \forall\; \lambda,\mu\in\Lambda,\; \sigma\in W,
	\end{align*}
	which yields $u = \sigma^{-1} (u)$ for all $\sigma \in W$. So $u \in (U^0_\flat)^W.$
\end{proof}

\begin{theorem} \label{Image in}
	Suppose rank $n$ is even, then $\varrho^{\lambda+\rho,\mu}  (\xi(z)) = \varrho^{\sigma(\lambda+\rho),\mu} (\xi(z)),\, \forall\; z$ $ \in Z(U),\; \sigma \in W,\; \lambda,\; \mu \in \Lambda.\,$
	As a result, we have $\operatorname{Im} (\xi) \subseteq (U^0_\flat)^W.$
\end{theorem}
\begin{proof}
	Let $z\in Z(U) $ and $ \mu \in \Lambda$, take a $ \lambda \in \Lambda$ such that  $(\lambda,\alpha_i^\vee) \geqslant 0$ for some fixed $i$. Let $v_{\lambda,\mu}$ be the highest weight vector of the Verma module $M(\varrho^{\lambda,\mu})$. Then
	\begin{align*}
		z\cdot v_{\lambda,\mu}
		= \pi(z) \cdot v_{\lambda,\mu}
		= \varrho^{\lambda,\mu}(\pi(z)) v_{\lambda,\mu}
		=\varrho^{\lambda+\rho,\mu}(\xi(z)) v_{\lambda,\mu}.
	\end{align*}
	That is, $z$ acts on $M(\varrho^{\lambda,\mu})$ by scalar $\varrho^{\lambda+\rho,\mu}(\xi(z)).\,$
	On the other hand, by \cite[Corollary 2.6]{BGH06} and \cite[Property 37]{PHR10}, let $[m]_i = \frac{r_i^m - s_i^m}{r_i - s_i}$, then we have
	
	\begin{align*}
		e_i f_i^{(\lambda,\alpha_{i}^\vee) +1 } \cdot v_{\lambda,\mu}
		&= [(\lambda,\alpha_i^\vee )+1]_i \; f_i^{(\lambda,\alpha_i^\vee)} \frac{r_i^{-(\lambda,\alpha_i^\vee) } \omega_i - s_i^{-(\lambda,\alpha_i^\vee ) }\omega'_i}{r_i - s_i} \cdot v_{\lambda,\mu}.
	\end{align*}		
	Notice that
	\begin{align*}
		(r_i^{-(\lambda,\alpha_i^\vee) } \omega_i - s_i^{-(\lambda,\alpha_i^\vee ) }\omega'_i) \cdot v_{\lambda,\mu}
		&= \left( r_i^{-(\lambda,\alpha_i^\vee)} \varrho^{\lambda,\mu} (\omega_i) - s_i^{-(\lambda,\alpha_i^\vee)} \varrho^{\lambda,\mu}(\omega'_i)
		\right) \cdot v_{\lambda,\mu}.
	\end{align*}		
	Since
	$r_i^{-(\lambda,\alpha_i^\vee)} \varrho^{\lambda,0} (\omega_i) = s_i^{-(\lambda,\alpha_i^\vee)} \varrho^{\lambda,0}(\omega'_i)
	,$
	it follows that
	\begin{align*}
		e_j f_i^{(\lambda,\alpha_i^\vee) +1 } \cdot v_{\lambda,\mu} &= 0, \hspace{5cm} j=1,\cdots,n, \\
		z f_i^{(\lambda,\alpha_i^\vee) +1 } \cdot v_{\lambda,\mu}
		&= \pi(z)f_i^{(\lambda,\alpha_i^\vee)+1} \cdot v_{\lambda,\mu} \\
		&= \varrho^{\sigma_i (\lambda + \rho) -\rho,\mu} (\pi (z)) f_i^{(\lambda,\alpha_i^\vee)+1} \cdot v_{\lambda,\mu} \\
		&= \varrho^{\sigma_i (\lambda + \rho),\mu} (\xi (z)) f_i^{(\lambda,\alpha_i^\vee)+1} \cdot v_{\lambda,\mu}, \hspace{1cm} \forall\; z \in Z(U). \\
	\end{align*}
	Hence, $z$ acts on $M(\varrho^{\lambda,\mu})$ by scalar $\varrho^{ \sigma_i(\lambda + \rho),\mu}(\xi(z)).\,$ we have
	\begin{align}\label{equ_rho_xi}
		\varrho^{\lambda+\rho,\mu}  (\xi(z)) = \varrho^{\sigma_i(\lambda+\rho),\mu} (\xi(z)).
	\end{align}
	
	In fact, equation (\ref{equ_rho_xi}) holds for any $\lambda \in \Lambda$.
	This is because if $(\lambda, \alpha_i^\vee)=-1,$ then $\lambda+\rho = \sigma_i(\lambda'+\rho)$ such that (\ref{equ_rho_xi}) holds.
	If $(\lambda, \alpha_i^\vee)<-1,$ let $\lambda' = \sigma_i(\lambda+\rho) - \rho,$ then $(\lambda',\alpha_i^\vee) \geqslant 0$ such that (\ref{equ_rho_xi}) holds for $\lambda'$. Relacing $\lambda'$ with $ \sigma_i(\lambda+\rho) - \rho$ into the result that (\ref{equ_rho_xi}) holds for $\lambda$ in this case.
	Finally, since (\ref{equ_rho_xi}) holds for each $\sigma_i$, so it holds for all $\sigma \in W$, which implies that Im $(\xi) \subseteq (U^0_\flat)^W$, by Lemma \ref{If in}.
\end{proof}

\section{Central elements and the Harish-Chandra theorem}

In this section, we will deal with arbitrary rank.

\subsection{The Harish-Chandra theorem}
In what follows, we aim to prove that the subspace $(U^0_\flat)^W$ is in $\operatorname{Im}(\xi)$.
\begin{lemma} \label{center-ad} Let $z \in U$, then
	$z \in Z(U) $ if and only if $\operatorname{ad}_l (x) z = \varepsilon (x) z,\, \forall\; x \in U. $
\end{lemma}
\begin{proof}
	Suppose $z \in Z(U)$, then for all $x\in U$, we have
	\begin{align*}
		\operatorname{ad}_l (x) z
		= \sum\limits_{(x)} x_{(1)} z S(x_{(2)})
		= z \sum\limits_{(x)} x_{(1)} S(x_{(2)})
		= \varepsilon (x) z.
	\end{align*}
	Conversely, if $\operatorname{ad}_l (x) z =  \varepsilon(x) z$ holds for all $x \in U$, then
	\begin{align*}
		\omega_i z \omega_i^{-1} = \operatorname{ad}_l (w_i) z =  \varepsilon(\omega_i) z = z.
	\end{align*}
	Since for each generator $e_i$ and $f_i$ of $U_{r,s}(\mathfrak{g})$, we have
	\begin{align*}
		0
		&= \varepsilon (e_i) z
		= \operatorname{ad}_l (e_i) z  \\
		&= e_i z + \omega_i z S(e_i)
		= e_i z - (\omega_i z \omega_i^{-1}) e_i
		= e_i z - z e_i,\\
		0
		&= \varepsilon (f_i) z
		= \operatorname{ad}_l (f_i) z \\
		&= zS(f_i) + f_i z S(\omega'_i)
		= (-zf_i +f_i z) {\omega'_i}^{-1}.	
	\end{align*}
	So $z \in Z(U).$
\end{proof}

\begin{lemma} \label{Rooso realize}	
	Given a bilinear form $\Psi: \; U^-_{-\mu} \times U^+_{\upsilon} \rightarrow \mathbb{K}$ and a pair $(\eta,\phi)\in Q \times Q$,
	then there exists an element $u \in U^-_{-\mu} U^0 U^+_{\upsilon}$ such that for any $x \in U^+_\upsilon,\; y \in U^-_{-\mu}$ and $ (\zeta,\psi) \in Q \times Q$,
	\begin{align*}
		\langle u \;|\; (y {\omega'_\mu}^{-1}) \omega'_\zeta \omega_\psi x \rangle = \langle \omega'_\zeta, \omega_\phi \rangle \langle \omega'_\eta,\omega_\psi \rangle \Psi(y,x).
	\end{align*}
\end{lemma}
\begin{proof}
	Let $\mu \in Q^+$ and $\{u^\mu_1,\cdots,u^\mu_{d_\mu}\}$ be a basis of $U^+_\mu$, then take a dual basis $\{ v^\mu_1,\cdots,v^\mu_{d_\mu} \}$  in $U^-_{-\mu}$ with respect to $\langle-,- \rangle$. Take that
	\begin{align*}
		u = \sum\limits_{i,j} (rs^{-1})^{-(\rho,\nu)} \hspace{0.1cm} \Psi(v^\mu_j, u^\nu_i) \hspace{0.1cm} (v^\nu_i {\omega'_\nu}^{-1}) \omega'_{\eta} \omega_\phi  u^\mu_j ,
	\end{align*}
	then $u$ satisfies the identity in the lemma.
\end{proof}

\begin{defi}
	$(1)$ Define a  $U$-module structure on $U^*$ by $$(x\cdot f)(v) = f(\operatorname{ad}_l(S(x)) v),\quad \forall\; x\in U,\; f\in U^*.$$ Then we define a morphism $\beta:\; U \longrightarrow U^*, \; u \mapsto \langle u \;|\; - \rangle.\,$ It follows that $\beta$ is an injective morphism of $U$-module since Rosso form is nondegenerate and ad-invariant.
	
	$(2)$  Let $M$ be a finite-dimensional $U$-module. For each $m\in M,\; f\in M^*,$ define the matrix coefficient
	by $C_{f,m}\in U^*,\; C_{f,m}(v) = f(v\cdot m), \,\forall \; v \in U.$
\end{defi}

\begin{proposition} \label{fdmod matrix coefficients}
	Let $M$ be a finite-dimensional $U$-module. Decompose $M=\bigoplus_{\lambda \in \Pi(M)} M_\lambda$, where
	\begin{align*}
		M_\lambda = \{m\in M \mid \omega_i \cdot m = \varrho^\lambda(\omega_i) m ,\ \omega'_i \cdot m = \varrho^\lambda(\omega'_i)m \}.
	\end{align*}
	If the weight set of $M$ satisfies $\operatorname{Wt}(M)\subseteq Q$, then for $f\in M^*,\; m\in M,$ there exists a unique $u\in U$ such that
	\begin{align*}
		C_{f,m}(v) =\langle u \;|\; v\rangle,\quad \forall\; v \in U.
	\end{align*}	
\end{proposition}
\begin{proof}
	We start by proving the existence of $u$. Since $C_{f,m}$ is linear with respect to $m \in M$, it is sufficient to check the case of $m \in M_\lambda$ for each $\lambda$. Suppose $v$ is a monomial $v = (y {\omega'_\mu}^{-1})\omega'_\eta \omega_\phi x,$ where $x \in U^+_\nu,\; y\in U^-_{-\mu}$, then for each $f \in M^*$, we have
	\begin{align*}
		C_{f,m}(v)
		&=C_{f,m}((y {\omega'_\mu}^{-1})\omega'_\eta \omega_\phi x)\\
		&=\langle \omega'_\eta , \omega_{-(\lambda+\nu)} \rangle \langle \omega'_{\lambda+\nu} , \omega_\phi \rangle f((y {\omega'_\mu}^{-1}).x.m).
	\end{align*}
	Since $\Psi:  (y,x)\mapsto f((y {\omega'_\mu}^{-1}).x.m)$ is a bilinear form, by Lemma \ref{Rooso realize}, there exists a unique $u_{\nu\mu}$ such that for all $v \in U^-_{-\mu} U^0 U^+_{\nu},$  $C_{f,m}(v) = \langle u_{\nu\mu} \;|\; v \rangle$ holds.
	
	More generally, let $v \in U$ with $v = \sum_{(\mu,\nu)} v_{\mu\nu}$, where $v_{\mu\nu} \in U^-_{-\mu} U^0 U^+_{\nu}$. Since $M$ is finite-dimensional, there exists a finite set $\Omega$ of pairs $(\mu,\nu) \in Q\times Q$, such that
	\begin{align*}
		C_{f,m} (v) = C_{f,m} \Bigl(\sum\limits_{(\mu,\nu)} v_{\mu\nu} \Bigr),\quad \forall\; \nu \in U.
	\end{align*}
	Let $u = \sum_{(\mu,\nu) \in \Omega} u_{\nu\mu}$, then
	\begin{align*}
		\langle u \;|\; v \rangle
		&= \sum\limits_{(\mu,\nu),(\mu',\nu') \in \Omega} \langle u_{\nu'\mu'} \;|\; v_{\mu\nu} \rangle \\
		&= \sum\limits_{(\mu,\nu) \in \Omega} \langle u_{\nu\mu} \;|\; v_{\mu\nu} \rangle \\
		&= \sum\limits_{(\mu,\nu) \in \Omega} \langle u_{\nu\mu} \;|\; v \rangle
		= C_{f,m} (v).
	\end{align*}
	Here the second row of the equations holds by Lemma \ref{nondegen for grading}.
\end{proof}

\begin{lemma}
	Let $(M,\zeta)$ be a weight module, where $\zeta: U \longrightarrow \operatorname{End}(M)$  and define a linear map $\Theta:\;M\longrightarrow M$ by
	\begin{align*}
		m\mapsto (rs^{-1})^{-(\rho,\lambda)} m, \quad \forall\; m\in M_\lambda, \; \lambda \in \Lambda.
	\end{align*}
	Then for all $u\in U,$ $ \Theta \circ \zeta(u)=\zeta(S^2(u)) \circ \Theta.$ That is,
	\begin{align*}
		\Theta(u.m) = S^2(u).\Theta(m), \quad \forall\; m\in M.
	\end{align*}
\end{lemma}
\begin{proof}
	It is sufficient to check it for the generators $e_i , f_i$. Notice that
	\begin{align*}
		\langle \omega'_i, \omega_i \rangle &= a_{ii} = (rs^{-1})^{d_i}=(rs^{-1})^{\frac{(\alpha_i,\alpha_i)}{2}}\\
		&=(rs^{-1})^{ (\varpi_i,\alpha_i) } =  (rs^{-1})^{(\rho,\alpha_{i})}\;,
	\end{align*}
	which yields for any $m \in M_\lambda,$
	\begin{align*}
		S^2(e_i). \Theta(m)
		&= (rs^{-1})^{-(\rho,\lambda)} (S^2(e_i).m) \\
		&= (rs^{-1})^{-(\rho,\lambda)} \langle \omega'_i, \omega_i \rangle^{-1} e_i . m \\
		&= (rs^{-1})^{-(\rho,\lambda+\alpha_i)} e_i . m  =  \Theta(e_i . m),\\
		S^2(f_i). \Theta(m)
		&= S^2 (f_i) (rs^{-1})^{-(\rho,\lambda)}.m \\
		&=\langle \omega'_i, \omega_i \rangle (rs^{-1})^{-(\rho,\lambda)} f_i . m \\
		&=(rs^{-1})^{-(\rho,\lambda-\alpha_i)} f_i . m = \Theta(f_i . m).
	\end{align*}
	
	This completes the proof.
\end{proof}

\begin{proposition} \label{Construct a central element}
	For $\lambda \in \Lambda^+,$ define a quantum trace $t_\lambda \in U^*$ as $t_\lambda (u) := \operatorname{tr}_{L(\lambda)} (u \Theta).\,$ If $\lambda \in \Lambda^+ \cap Q,$ then $t_\lambda \in \operatorname{Im} (\beta),$ and $z_\lambda := \beta^{-1}(t_\lambda) \in Z(U).$
\end{proposition}
\begin{proof}
	(1) Since $\beta(u) = \langle u \;|\; - \rangle$  is injective and $t_\lambda = \operatorname{tr}_{L(\lambda)} (-\circ \Theta) \in U^*,$ denote the dimension of $L(\lambda)$ by $d$, take a basis  $\{ m_i\}^{d}_{i=1}$ in $L(\lambda)$ and its dual basis $\{f_i\}^d_{i=1}$
	in $L(\lambda)^*$, we have
	\begin{gather*}
		v.\Theta(m_i) = \sum\limits_{j=1}^{d} f_j(v. \Theta (m_i))m_j = \sum\limits_{j=1}^{d} C_{f_j,\Theta(m_i)} (v) m_j ,\\
		t_\lambda (v) = \operatorname{tr}_{L(\lambda)} (v \circ \Theta) = \sum\limits_{i=1}^{d} C_{f_i,\Theta(m_i)} (v) .
	\end{gather*}
	By Theorem \ref{Rooso realize}, for each $i=1,\cdots,d,$ there exists $u_i \in U$ that realizes the matrix coefficient by $\langle u_i \;|\; v \rangle = C_{f_i,\Theta(m_i)} (v)$. Write $u = \sum_{i=1}^{d} u_i,$ then
	\begin{align*}
		\beta(u)(v) &=\langle u \;|\; v \rangle = \sum\limits_{i=1}^{d} \langle u_i \;|\; v\rangle \\
		&=\sum\limits_{i=1}^{d} C_{f_i,\Theta(m_i)} (v) = t_\lambda (v).
	\end{align*}
	It follows that $t_\lambda \in \operatorname{Im} (\beta)$.
	
	(2) Since $U^*$ has a $U$-module structure $(x\cdot f)(v) = f(\operatorname{ad}_l(S(x)) v),\, \forall\; x\in U,\; f\in U^*,\,$ then for any $x,u \in U,$ we have
	\begin{align*}
		(S^{-1}(x).t_\lambda) (u)
		&= t_\lambda (\operatorname{ad}_l (x) u)
		= \operatorname{tr}_{L(\lambda)} \big(\sum_{(x)} x_{(1)} u S(x_{(2)}) \Theta \big) \\
		&= \operatorname{tr}_{L(\lambda)} \big(u \sum_{(x)} S(x_{(2)}) \Theta x_{(1)} \big)
		= \operatorname{tr}_{L(\lambda)} \big(u \sum_{(x)} S(x_{(2)}) S^2(x_{(1)}) \Theta  \big) \\	
		&= \operatorname{tr}_{L(\lambda)} \big(u S \big(\sum_{(x)} S(x_{(1)}) x_{(2)} \big)\Theta  \big)
		= \varepsilon (x) \operatorname{tr}_{L(\lambda)} (u\Theta) \\
		&= \varepsilon (x) t_\lambda (u).
	\end{align*}
	Substituting $S(x)$ for $x$ above, we get $x.t_\lambda = \varepsilon (x) t_\lambda,\, \forall \; x \in U.$ Notice the fact $t_\lambda \in \operatorname{Im} \beta,$ we define  $z_\lambda := \beta^{-1}(t_\lambda),$ then
	\begin{equation*}
		\begin{split}
			&x.t_\lambda = x.\beta(z_\lambda)
			= (x. \beta)(z_\lambda)
			= \beta(\operatorname{ad}_l (x) (z_\lambda)), \\
			&\varepsilon (x) t_\lambda  = \varepsilon(x) \beta (z_\lambda)
			= \beta (\varepsilon(x) z_\lambda).
		\end{split}
	\end{equation*}
	Since $\beta$ is injective, $\operatorname{ad}_l (x) (z_\lambda) = \varepsilon(x) z_\lambda,\, \forall\; x \in U.$ Therefore, by Lemma \ref{center-ad}, it follows that $z_\lambda \in Z(U)$.
\end{proof}

As described in Theorem \ref{fdmod matrix coefficients}, for each $\lambda \in \Lambda$, there is the unique simple weight module $L(\lambda)$ with the weight space decomposition $\bigoplus_{\tau \leqslant \lambda} L(\lambda)_\tau.$ Define $d_\tau = \operatorname{dim} \hspace{0.1cm} L(\lambda)_\tau$ and set a basis $\{m_i^\tau\}_{i=1}^{d_\tau}$ for this weight space. Then $L(\lambda)$ has a basis $\{m_i^\tau\}_{\tau,i}$ and $L(\lambda)^*$ has the canonical dual basis $\{f_\tau^i\}_{\tau,i}.$ Since $L(\lambda)$ has finite dimension, there exists a finite set $\Omega \subseteq Q \times Q$ such that for any $v \in U = U^- U^0 U^+, m \in L(\lambda)$, we have $v\cdot m = \sum_{(\mu,\nu) \in \Omega} v_{\mu,\nu} \cdot m$, where $ v_{\mu,\nu} \in U^-_{-\mu} U^0 U^+_\nu$.

\begin{theorem} \label{Expressions of central elements}
	For each $\lambda \in \Lambda^+ \bigcap Q$, the central element $z_\lambda$ is
	\begin{align*}
		z_\lambda = \sum_{\tau \leqslant \lambda}
		\sum_{\mu \in Q^+}
		\sum_{i,j}  (rs^{-1})^{- (\rho,\tau+\mu)}
		\langle {\omega'_\mu}, \omega_{\tau+\mu} \rangle
		tr(v^\mu_j u^\mu_i \circ P_\tau) \hspace{0.1cm}
		v^\mu_i \omega'_{\tau} \omega^{-1}_{\tau+\mu} u^\mu_j.
	\end{align*}
	where $\{u^\mu_j\}_{j=1}^{d_\mu}$ is a basis of $U^+_\mu$ and  $\{v^\mu_i\}_{i=1}^{d_\mu}$ is the dual basis of $U^-_{-\mu}$ with respect to the restriction of $\langle-,- \rangle$ to $U^-_{-\mu} \times U^+_{\mu}$, and $P_\tau$ is the projector from $L(\lambda)$ to $L(\lambda)_\tau.$
\end{theorem}
\begin{proof}
	We have shown that for any $v\in U$,
	\[
	\operatorname{tr}_{L(\lambda)} (v \circ \Theta) = \sum_{\tau,l} f_\tau^l(v\cdot \Theta(m^\tau_l)) = \sum_{\tau,l} C_{f_\tau^l,\Theta(m_l^\tau)}(v).
	\]
	
	Firstly, we restrict $v$ to any graded space $U^-_{-\mu} U^0 U^+_{\nu}$ and take a monomial $v = (y {\omega'_\mu}^{-1}) \omega'_\eta \omega_\phi x,$ $x\in U^+_\nu, y\in U^-_{-\mu},$ then by Theorem \ref{fdmod matrix coefficients}, we have
	\begin{align*}
		C_{f_\tau^l,\Theta(m_l^\tau)}(v)
		&=f^l_\tau ( (y {\omega'_\mu}^{-1}) \omega'_\eta \omega_\phi x\cdot \Theta(m_l^\tau)) \\
		&=\langle \omega'_\eta , \omega_{-(\tau+\nu)} \rangle \langle \omega'_{\tau+\nu} , \omega_\phi \rangle f^l_\tau ((y {\omega'_\mu}^{-1})x \cdot \Theta(m^\tau_l)).
	\end{align*}
	Then by Lemma \ref{Rooso realize}, put $\Psi(y,x) = f^l_\tau((y {\omega'_\mu}^{-1})x \cdot \Theta(m^\tau_i))$, then we get an element
	\begin{align*}
		z^{(\tau,l)}_{\nu,\mu}
		&= \sum_{i,j} (rs^{-1})^{- (\rho,\nu)}
		\Psi(v_j^\mu,u_i^\nu) \hspace{0.1cm}
		v^\nu_i \omega'_{\tau} \omega_{-(\tau+\nu)} u^\mu_j\\
		&= \sum_{i,j} (rs^{-1})^{- (\rho,\tau+\nu)}
		f^l_\tau(v^\mu_j {\omega'_\mu}^{-1} u^\nu_i . m^\tau_l) \hspace{0.1cm}
		v^\nu_i \omega'_{\tau} \omega_{\tau+\nu}^{-1} u^\mu_j,
	\end{align*}
	so that $\langle z^{(\tau,l)}_{\nu,\mu} \;|\; v \rangle = C_{f_\tau^i,\Theta(m_i^\tau)}(v)$, for any $v \in U^-_{-\mu} U^0 U^+_{\nu}$.
	
	Further, we add up all the elements labeled by the finite set $\Omega$ and write $z^{(\tau,l)} = $ $ \sum_{(\nu,\mu) \in \Omega} $ $ z^{(\tau,l)}_{\nu,\mu}$, then
	\begin{align*}
		\langle z^{(\tau,l)} \;|\; v \rangle = C_{f_\tau^l,\Theta(m_l^\tau)}(v), \quad \forall\; v\in U.
	\end{align*}
	Finally, since $z_\lambda = \sum_{\tau,l} z^{(\tau,l)},$ we get the expression
	\begin{align*}
		z_\lambda = \sum_{\tau \leqslant \lambda}
		\sum_{l=1}^{d^\tau}
		\sum_{(\nu,\mu) \in \Omega}
		\sum_{(i,j)=(1,1)}^{(d_\nu,d_\mu)}
		(rs^{-1})^{- (\rho,\tau+\nu)}
		f^l_\tau(v^\mu_j {\omega'_\mu}^{-1} u^\nu_i \cdot m^\tau_l) \hspace{0.1cm}
		v^\nu_i \omega'_{\tau} \omega^{-1}_{\tau+\nu} u^\mu_j	.
	\end{align*}
	
	Since $\{m_i^\tau\}_{i=1}^{d_\tau}$ and $\{f_\tau^i\}_{\tau,i}$ are dual to each other, when $\nu \neq \mu$ the component in $z^{(\tau,l)}_{\nu,\mu}=0$. Also, we can simplify the expression of $z_\lambda$ by projectors $P_\tau$, then
	\begin{align*}
		z_\lambda = \sum_{\tau \leqslant \lambda}
		\sum_{\mu \in Q^+}
		\sum_{i,j}  (rs^{-1})^{- (\rho,\tau+\mu)}
		\langle {\omega'_\mu}, \omega_{\tau+\mu} \rangle
		tr(v^\mu_j u^\mu_i \circ P_\tau) \hspace{0.1cm}
		v^\mu_i \omega'_{\tau} \omega^{-1}_{\tau+\mu} u^\mu_j.
	\end{align*}		
	
	This completes the proof.
\end{proof}	

\begin{remark}
	The explicit expression of central elements of one-parameter quantum groups (e.g., see \cite{D23,J96}) is a special case of Theorem \ref{Expressions of central elements}. As in $U_q(g),$ for each $\lambda \in 2\Lambda \bigcap Q^+$ there exists an element $z_\lambda$ in the centre such that $\langle z_\lambda \;|\; -\rangle = \text{tr}_{L(\lambda)}(-.K^{-1}_{2\rho}),$ where
	\begin{align*}
		z_\lambda &=
		\sum_{\tau \in \operatorname{Wt}(L(\lambda))}
		\sum_{l=1}^{d^\tau}
		\sum_{\mu \in Q^+}
		\sum_{i,j=1}^{d_\mu}
		q^{-2 (\rho,\tau+\mu)} \hspace{0.1cm}
		f^l_\tau(v^\mu_j K_\mu u^\mu_i \cdot m^\tau_l) \hspace{0.1cm}
		(v^\mu_i K_\mu) K_{-2(\tau+\mu)} u^\mu_j.	\\
		&= \sum_{\tau \leqslant \lambda}
		\sum_{\mu \in Q^+}
		\sum_{i,j}
		q^{-2 (\rho,\tau+\mu)} \hspace{0.1cm}
		\text{tr}(v^\mu_j K_\mu u^\mu_i \circ P_\tau) \hspace{0.1cm}
		v^\mu_i K_{-2\tau-\mu} u^\mu_j.	
	\end{align*}
	When $r=q,\; s=q^{-1},\; \omega_i = K_i,\; \omega'_i= K_i^{-1}$, two expressions above are the same.
\end{remark}

\begin{theorem} \label{z*}
	Suppose rank $n$ is odd, there is one extra invertible central generator $z_{*}:= \prod_{i=1}^{n} (\omega_i \omega'_i)^{v^*_i}$ in $U_{r,s}(\mathfrak{g})$ $($which doesn't survive in $U_q(\mathfrak g)$$)$, where the vector $v^*$ is given in Proposition \ref{InvR+S}.
\end{theorem}
\begin{proof}
	It suffices to show $z_{*}$ commutes with generators $e_i$ and $f_i$.
	\begin{align*}
		z_{*} e_i
		=\big(\prod_{j=1}^{n} (\omega_j \omega'_j)^{v^*_j} \big) e_i
		&= e_i \big(\prod_{j=1}^{n} (\omega_j \omega'_j)^{v^*_j}\big) \prod_{j=1}^{n}  (a_{ji} a^{-1}_{ij})^{v^*_j}\\
		&= e_i z_{*} (rs)^{ \sum_{j=1}^n (\langle i,j \rangle - \langle j,i \rangle) v^*_j}  \\
		&= e_i z_{*} (rs)^{(R+S) v^{*}} = e_i z_{*},
	\end{align*}
	\begin{align*}
		z_{*} f_i
		&= f_i \big(\prod_{j=1}^{n} (\omega_j \omega'_j)^{v^*_j}\big) \prod_{j=1}^{n}  	(a_{ji} a^{-1}_{ij})^{-v^*_j}\\
		&= f z_{*} (rs)^{-(R+S) v^{*}} = f_i z_{*}.
	\end{align*}		
	
	This completes the proof.
\end{proof}

\begin{proposition} \label{z*out}
	When rank $n$ is odd, the central element $z_{*}$ is a fixed point of the Harish-Chandra homomorphism $\xi$, and $\xi(z_{*}) = z_{*} \notin (U^0_\flat)^W.$
\end{proposition}
\begin{proof}
	The first statement is proved directly as follows
	\begin{align*}
		\xi(z_*)
		&= (\prod_{j=1}^n \varrho^{-\rho} ( (\omega_j \omega'_j)^{v^*_j})) z_*
		= (\prod_{i,j=1}^n  (a_{ji}^{-1} a_{ij})^{v^*_j/2})  z_* \\
		&= (rs)^{\frac{1}{2} \sum_{i,j=1}^{n} (R+S)_{ij} v^*_j}  z_*
		= z_*.
	\end{align*}
	Since $U_{r,s}$ is also a $Q$-bigraded Hopf algebra \cite{HP12} where $e_i \in (U_{r,s})_{(\alpha_i,0)}$, $f_i \in (U_{r,s})_{(0,-\alpha_{i})}$, $\omega_i, \omega'_i \in (U_{r,s})_{(\alpha_i,-\alpha_i)}$. Thus all generators $\omega'_\eta \omega_{-\eta}$ of $U^0_\flat$ have the same bigrade $(0,0)$, they can not generate an element graded by $(\eta,-\eta), \eta \in Q \setminus {0},$ which leads to the second statement.
\end{proof}

\begin{theorem} \label{HC iso}
	If rank $n$ is even, then the Harish-Chandra homomorphism $\xi:\; Z(U_{r,s}(\mathfrak{g})) \longrightarrow (U^0_\flat)^W$ is an algebra isomorphism. If rank $n$ is odd, we have $\xi$ is injective and $  \operatorname{Im}(\xi) \supseteq (U^0_\flat)^W \otimes \mathbb{K}[z_{*},z_{*}^{-1}]$.
\end{theorem}
\begin{proof}
	Firstly by Theorem \ref{Expressions of central elements}, for each $z_\lambda,\; \lambda \in \Lambda^+ \cap Q,$ we have
	\begin{align*}
		z_{\lambda}^0
		&= \sum_{\mu \leqslant \lambda} (rs^{-1})^{- (\rho,\mu)} \operatorname{dim}(L(\lambda)_\mu) \omega'_\mu \omega_{-\mu}.
	\end{align*}
	Therefore, by definition in section 2.5, we have
	\begin{align*}
		\xi(z_\lambda)
		&= \gamma^{-\rho}(z_{\lambda}^0)
		= \sum_{\mu \leqslant \lambda} (rs^{-1})^{- (\rho,\mu)} \operatorname{dim}(L(\lambda)_\mu) \hspace{0.1cm} \varrho^{-\rho}(\omega'_\mu \omega_{-\mu}) \omega'_\mu \omega_{-\mu}\\
		&= \sum_{\mu \leqslant \lambda} \operatorname{dim}(L(\lambda)_\mu)  \omega'_\mu \omega_{-\mu}.
	\end{align*}
	Notice that $\{ \xi(z_\lambda) \mid \lambda \in \Lambda^+ \cap Q\} \subseteq \operatorname{Im} \xi \subseteq(U^0_\flat)^W$ in Theorem \ref{If in}. It is sufficient to show $(U^0_\flat)^W \subseteq \operatorname{Im}(\xi).$ Since for each $\eta \in Q,$ there exists a unique $\sigma \in W$ such that $\sigma(\eta) \in \Lambda^+ \cap Q,$ it is clear that the elements
	\begin{align*}
		\operatorname{av}(\lambda) = \frac{1}{|w|} \sum_{\sigma \in W} \sigma(\omega'_\lambda \omega_{-\lambda})
		= \frac{1}{|w|} \sum_{\sigma \in W} \omega'_{\sigma(\lambda)} \omega_{-\sigma(\lambda)},\, \lambda \in \Lambda^+ \cap Q
	\end{align*}
	form a basis of $(U^0_\flat)^W$.
	
	It only remains to prove that $\operatorname{av}(\lambda) \in  \operatorname{Im} (\xi).$
	By induction on the height of $\lambda$,
	if $\lambda = 0$, then $\operatorname{av}(\lambda) = 1 \in \operatorname{Im} (\xi).$ Assume that $\lambda >0 $, by the fact that $\operatorname{dim}(L(\lambda)_\mu) = \operatorname{dim}(L(\lambda)_{\sigma(\mu)}),\, \forall \; \sigma\in W$ and $\operatorname{dim}(L(\lambda)_\lambda) = 1$, we have
	\begin{align*}
		\xi(z_\lambda) &= \sum_{\mu \leqslant \lambda} \operatorname{dim}(L(\lambda)_\mu)  \omega'_\mu \omega_{-\mu} \\
		&= |W| \operatorname{av}(\lambda) + |W| \sum_{\mu < \lambda,\; \mu \in \Lambda^+ \cap Q} \operatorname{dim}(L(\lambda)_\mu) \operatorname{av}(\mu).
	\end{align*}		
	By the induction hypothesis, we get
	\begin{align*}
		\operatorname{av}(\lambda) &= \frac{1}{|W|} \xi(z_\lambda) - \sum_{\mu < \lambda,\; \mu \in \Lambda^+ \cap Q}\operatorname{dim}(L(\lambda)_\mu) \operatorname{av}(\mu) \in \operatorname{Im} \xi.
	\end{align*}
	Therefore,
	\begin{align*}
		(U^0_\flat)^W = \operatorname{span}_\mathbb{K} \{ \operatorname{av}(\lambda) \;|\; \lambda \in \Lambda^+ \cap Q\} \subseteq \operatorname{Im} (\xi).
	\end{align*}
	
	When rank $n$ is even, combining Theorem \ref{Image in}, we have that $\xi$ is isomorphic to its image. When rank $n$ is odd, we have $ (U^0_\flat)^W \otimes \mathbb{K}[z_{*},z_{*}^{-1}] \subseteq \operatorname{Im}(\xi)$ by Theorem \ref{z*} and Proposition \ref{z*out}.
\end{proof}

\subsection{Alternative approach to central elements}
Similar to \cite{ZGB91}, we can also construct central elements by taking quantum partial trace.
\begin{proposition}
	Let $\lambda \in \Lambda^+$ and $\zeta: U_{r,s}(\mathfrak{g}) \longrightarrow \operatorname{End}(L(\lambda))$ be the weight representation. If there is an operator $\Gamma \in U_{r,s}(\mathfrak{g}) \otimes \operatorname{End}(L(\lambda))$ such that
	\begin{align*}
		\Gamma \circ (id \otimes \zeta) \Delta(x) = (id \otimes \zeta)\Delta(x)  \circ \Gamma, \hspace{0.4cm} \forall\; x \in U_{r,s}(\mathfrak{g}),
	\end{align*}		
	then the element $c = \operatorname{tr}_2 (\Gamma (1\otimes \Theta)) \in Z(U_{r,s}(\mathfrak{g}))$.
\end{proposition}
\begin{proof}
	It is enough to check this element commutes with all generators of $U$. In convenience we view the operator $(id \otimes \zeta)\Delta(x)$ as $\Delta(x),$ and write $\Gamma = \sum_{\Gamma} \Gamma_{(1)} \otimes \Gamma_{(2)}$ . Since $[\Gamma, \Delta(\omega_i^{\pm 1})]= 0,$ and $S^2(u) \Theta = \Theta u,\, \forall\; u \in U$, we have
	\begin{align*}
		0 & =\operatorname{tr}_2 ([\Gamma,\Delta(\omega_i^{\pm 1})] \cdot (1 \otimes \Theta \omega_i^{\mp 1})) \\
		&=  \sum_{\Gamma} \Gamma_{(1)} \omega_i^{\pm 1} \otimes \operatorname{tr}( \Gamma_{(2)} \omega_i^{\pm 1} \Theta \omega_i^{\mp 1}) - \sum_{\Gamma}  \omega_i^{\pm 1} \Gamma_{(1)} \otimes \operatorname{tr} (\omega_i^{\pm 1} \Gamma_{(2)} \Theta \omega_i^{\mp 1})  \\
		&=  \sum_{\Gamma} \Gamma_{(1)} \omega_i^{\pm 1} \otimes \operatorname{tr}( \Gamma_{(2)}  \Theta ) - \sum_{\Gamma}  \omega_i^{\pm 1} \Gamma_{(1)} \otimes \operatorname{tr} ( \Gamma_{(2)} \Theta )  \\
		&= \sum_{\Gamma} [\Gamma_{(1)},\omega_i^{\pm 1}] \otimes \operatorname{tr} ( \Gamma_{(2)} \Theta )\\
		&=[\operatorname{tr}_2 (\Gamma (1 \otimes \Theta)),\omega_i^{\pm 1}]=[c,\omega_i^{\pm 1}].
	\end{align*}
	Similarly, we have $[c,{\omega'_i}^{\pm 1}] = 0$.
	
	To prove $[c,e_i] = 0,$ we have to check $[c,e_i \omega_i^{-1}] = 0,$
	\begin{align*}
		0 & =\operatorname{tr}_2 ([\Gamma,\Delta(e_i \omega_i^{-1})] \cdot (1 \otimes \Theta \omega_i)) \\
		&= \operatorname{tr}_2 ([\Gamma,e_i \omega_i^{-1} \otimes\omega_i^{-1} +1 \otimes e_i \omega_i^{-1}](1\otimes \Theta \omega_i) ) \\
		&= \operatorname{tr}_2 (\Gamma(e_i \omega_i^{-1}) \otimes \Theta + \Gamma(1\otimes e_i \Theta) \\
		& \hspace{1cm} -(e_i \omega_i^{-1} \otimes \omega_i^{-1}) \Gamma (1\otimes \Theta\omega_i) -(1\otimes e_i \omega_i^{-1})\Gamma(1\otimes \Theta \omega_i) ) \\
		&= \operatorname{tr}_2 ( \Gamma(1\otimes \Theta)(e_i \omega_i^{-1} \otimes 1) + \Gamma(1\otimes e_i \Theta) \\
		&\hspace{1cm} -(e_i \omega_i^{-1} \otimes 1)\Gamma(1\otimes\Theta) - \Gamma(1\otimes \Theta \omega_i e_i w_i^{-1})  )\\
		&= [\operatorname{tr}_2 (\Gamma(1\otimes\Theta)),e_i \omega_i^{-1}] = [c,e_i \omega_i^{-1}].
	\end{align*}		
	The fifth equation holds since $\Theta \omega_i e_i w_i^{-1} = \omega_i S^2(e_i) w_i^{-1} \Theta = e_i \Theta$.
	
	Similarly, we have $[c,f_i] = 0$.
\end{proof}

\section{Centre of $\breve{U}_{r,s}(\mathfrak{g})$ of weight lattice type}
In this section, we will add some group-like elements to $U_{r,s}(\mathfrak{g})$ to get the two-parameter quantum group $\breve{U}_{r,s}(\mathfrak{g})$ of weight lattice type.

\begin{defi}
	The algebra $\breve{U}_{r,s}(\mathfrak{g})$ of the so-called weight lattice form of $U_{r,s}(\mathfrak{g})$, is the unital associative algebra generated by elements $e_i,$ $f_i,$  $\omega_{\varpi_i}^{\pm 1},$ ${\omega'_{\varpi_i}}^{\pm 1}\;$ $(i=1,\cdots,n)$ over $\mathbb{K}$. Set $\omega_i := \prod_{l=1}^n {\omega_{\varpi_l}}^{c_{li}},\; \omega'_j := \prod_{k=1}^n {\omega'_{\varpi_k}}^{c_{kj}}$, the generators satisfy the following relations:
	
	\hspace{1cm}
	
	$(X1,\Lambda)$ \,
	$\omega_{\varpi_i}^{\pm 1} {\omega'_{\varpi_j}}^{\pm 1} =  {\omega'_{\varpi_j}}^{\pm 1}\omega_{\varpi_i}^{\pm 1},
	\qquad \qquad
	\omega_{\varpi_i} \omega_{\varpi_i}^{-1} = 1= \omega'_{\varpi_j} {\omega'_{\varpi_j}}^{-1},$
	
	$(X2,\Lambda)$ \,
	$\omega_{\varpi_i} e_{j} \omega_{\varpi_i}^{-1} =
	r^{\langle j , \varpi_i \rangle} s^{-\langle \varpi_i , j \rangle}e_{j},
	\qquad \omega_{\varpi_i} f_{j} \omega_{\varpi_i}^{-1} =
	r^{-\langle j , \varpi_i \rangle} s^{\langle \varpi_i , j \rangle} f_{j},$
	
	\hspace{1.3cm} $\omega_{\varpi_i}^{\prime} e_{j} \omega_{\varpi_i}^{\prime-1} =
	r^{-\langle \varpi_i , j\rangle} s^{\langle j,\varpi_i\rangle} e_{j},
	\qquad \omega_{\varpi_i}^{\prime} f_{j} \omega_{\varpi_i}^{\prime-1} =r^{\langle \varpi_i , j\rangle} s^{-\langle j,\varpi_i\rangle} f_{j},$
	
	$(X3)$ \,
	$[e_i, f_j] = \delta_{ij} \frac{\omega_i -\omega_i'}{r_i - s_i},$

	$(X4)$ \,
	$(\operatorname{ad}_l e_i)^{1-c_{ij}}(e_j) = 0,
	\qquad (\operatorname{ad}_r f_i)^{1-c_{ij}}(f_j) = 0, \qquad (i\neq j).$
	
	\hspace{1cm}
	
	Here $\breve{U}_{r,s} (\mathfrak{g})$ naturally extends the Hopf structure of $U_{r,s} (\mathfrak{g})$ (elements $\omega_{\varpi_i}, {\omega'_{\varpi_i}}$ are group-likes).
\end{defi}	

\begin{proposition}
	There exists a unique skew-dual pairing $ \langle \; -,- \, \rangle: \breve{\mathcal{B}}' \times \breve{\mathcal{B}} \longrightarrow \mathbb{K}$ of the Hopf subalgebra $\breve{\mathcal{B}}$ and $\breve{\mathcal{B}}'$ satisfying
	\begin{align*}
		& \langle f_i , e_j \rangle = \delta_{ij} \frac{1}{s_i - r_i },
		\hspace{4.9cm}   1 \leqslant i , j \leqslant n, \\
		& \langle \omega'_{\varpi_i} , \omega_{\varpi_j} \rangle =
		r^{\langle \varpi_i , \varpi_j \rangle} s^{-\langle \varpi_j , \varpi_i \rangle},
		\hspace{3.2cm}   1 \leqslant i , j \leqslant n, \\
		& \langle {\omega'_i}^{\pm 1} , \omega_j^{-1} \rangle = \langle {\omega'_i}^{\pm 1} , \omega_j \rangle ^{-1} = \langle \omega'_i , \omega_j \rangle^{\mp 1},
		\hspace{1.8cm} 1 \leqslant i,j \leqslant n,
	\end{align*}
	and all other pairs of generators are 0.
\end{proposition}

\begin{coro}
	For any $\lambda_1 , \lambda_2 \in \Lambda,$ we have
	$\langle \omega'_{\lambda_1} , \omega_{\lambda_2} \rangle =
	r^{\langle \lambda_1 , \lambda_2 \rangle} s^{-\langle \lambda_2 , \lambda_1 \rangle}.$
\end{coro}

It is obvious that $U_{r,s}$ is a Hopf subalgebra of $\breve{U}_{r,s},$ and these two Hopf algebras have the same category of the weight modules, namely $\mathcal{O}^{r,s}_f$.
Notice that our approach to the Harish-Chandra isomorphism of $\breve{U}$ is exactly the same as in the previous sections, so it is sufficient to show that the lemmas and properties affected by the extension still hold.

\begin{lemma}
	The Rosso form $\langle-\;|\;-\rangle$ on $\breve{U}$ as in Definition \ref{Rosso} is non-degenerate.
\end{lemma}
\begin{proof}
	The only difference between Theorem \ref{nondegen} and the current result is the character. Define a group homomorphism $\chi_{\eta,\phi}: \Lambda \times \Lambda \longrightarrow \breve{\mathbb{K}}^\times$ by
	\begin{align*}
		\chi_{\eta,\phi} (\nu,\mu) = \langle \omega'_\eta, \omega_{\mu}  \rangle\langle \omega'_{\nu},\omega_{\phi} \rangle.
	\end{align*}
	Parallel to Lemma \ref{chi}, we only need to prove if $\chi_{\eta, \phi} = \chi_{\eta', \phi'},$ then $(\eta,\phi) = (\eta',\phi').$
	
	Let $\eta  = \sum_{i=1}^{n} \eta_i \alpha_i ,\; \eta'  = \sum_{i=1}^{n} \eta'_i \alpha_i,$ for $j =1,\cdots,n$, we have
	\begin{align*}
		1=\frac{\chi_{\eta,\phi}(0,\varpi_j)}{\chi_{\eta',\phi'}(0,\varpi_j)}
		&= \prod_{i=1}^{n}  \langle \omega'_i ,\omega_{\varpi_j} \rangle^{\eta_i - \eta'_i}
		=\prod_{i=1}^{n} \prod_{k=1}^n \langle \omega_{i}^{\prime}, \omega_{k} \rangle^{{c}^{kj} (\eta_i - \eta'_i)} \\
		&
		=r^{\sum_{i,k=1}^n {c}^{kj} R_{ki}(\eta_i - \eta'_i)} s ^{\sum_{i,k=1}^n {c}^{kj}S_{ki}(\eta_i - \eta'_i)},
	\end{align*}		
	which leads to ${C}^{-T} (R-S) (\eta - \eta') = 0$ and yields $\eta = \eta'$. A similar argument leads to $\phi = \phi'$.
\end{proof}
Repeating the proof of Theorem \ref{inj if}, one gets
\begin{theorem}
	The Harish-Chandra homomorphism $\breve{\xi} : Z(\breve{U}) \longrightarrow \breve{U}^0$ is injective.
\end{theorem}

Notice that the proof of Lemma \ref{nondegen of varho} is independent of where the coefficients $\eta_i,\; \phi_i$ are taken from, it follows that Lemma \ref{nondegen of varho} also hold true for $\breve{U}$.
The same reason can be used to show that Lemmas \ref{nondegen of varho}-\ref{If in}, and \ref{Rooso realize} still hold in $\breve{U}$, which leads to Theorem \ref{Image in}, that is $\operatorname{Im} (\breve{\xi}) \subseteq (\breve{U}^0_\flat)^W,$ where $\breve{U}^0_\flat := \bigoplus_{\eta \in \Lambda} \mathbb{K} \omega'_\eta \omega_{-\eta}$. Parallel to Proposition \ref{Construct a central element} and Theorem \ref{Expressions of central elements}, there are enough group-like elements to define all $z_\lambda \in \breve{U}, \, \forall\; \lambda \in \Lambda^+$.

\begin{example}
	Let $\mathfrak{g} = \mathfrak{sl}_2,\; \lambda = \varpi_1 = \frac{1}{2} \alpha_1$, then $z_* = \omega'\omega \in Z(U_{r,s}(\mathfrak{sl}_2))$, and the central element $z_{\varpi_1} \in Z(\breve{U}_{r,s}(\mathfrak{sl}_2))$ is given by
	\begin{equation*}
		\begin{aligned}
			z_{\varpi_1} &= \sum_{\tau=\pm \frac{1}{2} \alpha_1}
			\sum_{\mu=0, \alpha_1}
			(rs^{-1})^{- (\rho,\tau+\mu)}
			\langle {\omega'_\mu}, \omega_{\tau+\mu} \rangle
			tr(v^\mu u^\mu \circ P_\tau) \hspace{0.1cm}
			v^\mu \omega'_{\tau} \omega^{-1}_{\tau+\mu} u^\mu\\
			&= r^{-\frac{1}{2}} s^{\frac{1}{2}}  {\omega'}^{\frac{1}{2}} \omega^{-\frac{1}{2}}
			+ r^{\frac{1}{2}} s^{-\frac{1}{2}}  {\omega'}^{-\frac{1}{2}} \omega^{\frac{1}{2}}
			+ (s-r)^2 r^{-\frac{1}{2}} s^{-\frac{1}{2}}  f {\omega'}^{-\frac{1}{2}} \omega^{-\frac{1}{2}} e\\
			&= (s-r)^2 (rs)^{-\frac{1}{2}} \; z_*^{-\frac{1}{2}} \mathcal{C},
		\end{aligned}
	\end{equation*}
	where $\mathcal{C} = fe  + \frac{r\omega + s \omega'}{(s-r)^2} \in Z(U_{r,s}(\mathfrak{sl}_2))$ is the Casimir element of $U_{r,s}(\mathfrak{sl}_2)$, however, this central element $z_{\varpi_1} \notin U_{r,s}(\mathfrak{sl}_2)$. Further, Benkart et al. showed that $Z(U_{r,s}(\mathfrak{sl}_2))=\mathbb{K}[z_*,z_*^{-1}]\otimes\mathbb{K}[\mathcal{C}]$. This is because they could view $U_{r,s}(\mathfrak{sl}_{n+1})$ as a subalgebra of $U_{r,s} (\mathfrak{gl}_{n+1})$ and prove that $\xi(Z(U_{r,s}(\mathfrak{sl}_2))) \subseteq (U^0_\natural)^W$ based on extra Cartan generators of $U_{r,s} (\mathfrak{gl}_{n+1})$ \textnormal{(}see \cite{BKL06}\textnormal{)}, and luckily $\sigma_1(\omega'_\eta \omega_\phi) = \omega'_\phi \omega_\eta$, then
	\begin{equation*}
		\begin{split}
			(U^0_\natural)^W &=  \Bigl\{\,x=\sum_{\eta,\phi \in Q} k_{\eta,\phi} \omega'_\eta \omega_\phi \;\Big|\; k_{\eta,\phi} = k_{\phi,\eta},\; \forall\; \eta,\phi \in Q\,\Bigr\}\\
			&= \mathbb{K}[(\omega'\omega)^{\pm 1}, (\omega+\omega')] = \mathbb{K}[\xi(z_*)^{\pm 1}, \xi(\mathcal{C})] =\xi(\mathbb{K}[z_*^{\pm 1}, \mathcal{C}]),
		\end{split}
	\end{equation*}
	where the third equation holds by $\xi(\mathcal{C})=\frac{r \gamma^{-\frac{1}{2} \alpha_1}(\omega) + s \gamma^{-\frac{1}{2} \alpha_1}(\omega') }{(s-r)^{2}}= \frac{(rs)^{\frac{1}{2}}}{(s-r)^2} (\omega+\omega')$. Thus, the centre $Z(U_{r,s}(\mathfrak{sl}_2))$ has to be $\mathbb{K}[z_*,z_*^{-1}]\otimes\mathbb{K}[\mathcal{C}]$ since $\xi$ is injective. Similarly, we have $Z(\breve{U}_{r,s}(\mathfrak{sl}_2)) = \mathbb{K}[z_*^{\frac{1}{2}},z_*^{-\frac{1}{2}}]\otimes\mathbb{K}[\mathcal{C}]=\mathbb{K}[z_*^{\frac{1}{2}},z_*^{-\frac{1}{2}}]\otimes\mathbb{K}[z_{\varpi_1}]$.
	
	Let $\lambda = \alpha_1 \in \Lambda^+ \bigcap Q$, then
	$z_{\alpha_1} \in Z(U_{r,s})\subset Z(\breve{U}_{r,s})$ is given by
	\begin{equation*}
		\begin{aligned}
			z_{\alpha_1} &= \sum_{\tau=0,\pm \alpha_1}
			\sum_{\mu=0, \alpha_1, 2\alpha_1}
			(rs^{-1})^{- (\rho,\tau+\mu)}
			\langle {\omega'_\mu}, \omega_{\tau+\mu} \rangle
			tr(v^\mu u^\mu \circ P_\tau) \hspace{0.1cm}
			v^\mu \omega'_{\tau} \omega^{-1}_{\tau+\mu} u^\mu\\
			&= 1+ r^{-1} s \; \omega' \omega^{-1} +  r s^{-1}  {\omega'}^{-1} \omega \\
			& \hspace{0.7cm}+ (rs)^{-1} (s+r)(s-r)^2  f({\omega'}^{-1} +  \omega^{-1})e + (rs)^{-1}(s-r)^4  f^2 {\omega'}^{-1} \omega^{-1} e^2\\
			&= \frac{(s-r)^4}{rs} z_*^{-1} \mathcal{C}^2 - 1
			= z_{\varpi_1}^2 - 1 ,
		\end{aligned}
	\end{equation*}
	which also corresponds to the decomposition $L(\varpi_1)^{\otimes 2} \cong L(\alpha_{1}) \oplus L(0)$.
\end{example}

\begin{theorem} \label{z*^}
	When rank $n$ is odd, the element $z_{*}^{\frac{1}{\ell}}:= \prod_{i=1}^{n} (\omega_i \omega'_i)^{\frac{1}{\ell} v^*_i}$ is well-defined in $\breve{U}_{r,s}(\mathfrak{g})$, where $\ell=2$, except $\ell=4$ for $D_{2k+1}$, and the vector $v^*$ is given in Proposition \ref{InvR+S}.
\end{theorem}
\begin{proof}
	Since the element $z_{*}^{\frac{1}{\ell}} = \prod_{j=1}^{n} (\omega_j \omega'_j)^{\frac{1}{\ell} v^*_j} = \omega_{\varkappa} \omega'_{\varkappa}$, where
	\begin{equation*}
		\varkappa = \frac{1}{\ell} \sum_{j=1}^{n} \alpha_j v^*_j  = \frac{1}{\ell} \sum_{i,j=1}^{n} \varpi_i  C_{ij} v^*_j  = \sum_{i=1}^{n} \left(\frac{1}{\ell} \sum_{j=1}^{n} C_{ij} v^*_j\right) \varpi_i,
	\end{equation*}
	if $z_{*}^{\frac{1}{\ell}} \in \breve{U}_{r,s}(\mathfrak{g})$, then the coefficient $\frac{1}{\ell} \sum_{j=1}^{n} C_{ij} v^*_j \in \mathbb{Z}$ for each $i=1, 2, \cdots, n$. This is true since $C v^*$ satisfies
	\begin{center}
		\begin{tabular}{|c|l|}
			\hline
			$A_{2k+1}: \; \mathfrak{sl}_{2k+2}$&
			$C v^{*}=(2,-2,2,-2,\cdots,2,-2,2)^T$\\
			$B_{2k+1}:\; \mathfrak{so}_{4k+3}$&
			$C v^{*}=(2,-2,2,-2,\cdots,2,-2,2)^T$\\						
			$C_{2k+1}:\; \mathfrak{sp}_{4k+2}$&
			$C v^{*}=(4,-4,4,-4,\cdots,4,-4,2)^T$\\	
			$D_{2k+1}:\; \mathfrak{so}_{4k+2}$&
			$C v^{*}=(4,-4,4,-4,\cdots,4,-4,0)^T$\\		
			$E_7$&
			$C v^{*}=(2,-2,2,-2,-2,0,0)^T$\\	
			\hline										
		\end{tabular}
	\end{center}	
\end{proof}

Finally, by the fact that $(\breve{U}^0_\flat)^W $ has a basis $\{ \operatorname{av}(\lambda) \mid \lambda \in \Lambda^+\}$, repeating the proof of Theorem \ref{HC iso}, one gets:
\begin{theorem}
	The Harish-Chandra homomorphism $\breve{\xi}: Z(\breve{U}_{r,s}(\mathfrak{g})) \longrightarrow (\breve{U}^0_\flat)^W$ is an isomorphism of algebras when rank $n$ is even. In particular, for each $\lambda \in \Lambda^+$,
	\begin{align*}
		\breve{\xi}(z_\lambda) = \sum_{\mu \leqslant \lambda} \operatorname{dim} (L(\lambda)_\mu) \, \omega'_\mu \omega_{-\mu}.
	\end{align*}
	When rank $n$ is odd, we have $\breve{\xi}$ is injective and $\operatorname{Im}(\breve{\xi}) \supseteq (\breve{U}^0_\flat)^W \otimes \mathbb{K}[z_*^{\frac{1}{\ell}}, z_*^{-\frac{1}{\ell}}]$, where $\ell=2$, except $\ell=4$ for $D_{2k+1}$.
\end{theorem}

Now we are able to construct the character map to study the centre $Z(\breve{U}_{r,s}(\mathfrak{g}))$.
\begin{proposition}
	Let $K({U}_{r,s}) := {Gr}(\mathcal{O}^{r,s}_f) \otimes_\mathbb{Z} \mathbb{K},$ then the character map $\operatorname{Ch}_{r,s}: K({U}_{r,s}) \longrightarrow (\breve{U}^0_\flat)^W$ is an isomorphism of algebras with
	\begin{align*}
		\operatorname{Ch}_{r,s}([V]) = \sum_{\mu \leqslant \lambda} \operatorname{dim} (V_\mu) \, \omega'_\mu \omega_{-\mu}, \quad \forall \,  V \in \mathcal{O}^{r,s}_f.
	\end{align*}
\end{proposition}
\begin{proof}
	Since $ (V\otimes W)_\mu = \bigoplus_{\lambda\in\Lambda} V_\lambda \otimes W_{\mu-\lambda}$ still holds in $\mathcal{O}^{r,s}_f$, it is clear that $\operatorname{Ch}_{r,s}$ is a homomorphism of algebras. Next, if $\operatorname{Ch}_{r,s}([V])=\operatorname{Ch}_{r,s}([W])$ for some $V,W \in \mathcal{O}^{r,s}_f,$ then
	\begin{align*}
		\sum_{\mu \in \Pi(V)} \operatorname{dim} (V_\mu) \, \omega'_\mu \omega_{-\mu} = \sum_{\nu \in \Pi(W)} \operatorname{dim} (W_\nu) \, \omega'_\nu \omega_{-\nu},
	\end{align*} where $\Pi(V)$ is the weight set of $V$.
	Since $\omega'_\mu \omega_{-\mu}$ with $\mu \in \Lambda$ are linear independent, we have $\operatorname{dim}(V_\mu) = \operatorname{dim}(W_\mu)$ for any $\mu,$ i.e., $[V] = [W].$ Finally, for any $\lambda \in \Lambda^+,$ we have $\operatorname{Ch}_{r,s}([L(\lambda)]) = \breve{\xi}(z_\lambda)$. Parallel to the proof of Theorem \ref{HC iso}, $\operatorname{Ch}_{r,s}$ is also surjective.
\end{proof}

In \cite{HP12}, Hu and Pei have studied the deformation theory of the representations of two-parameter quantum groups $U_{r,s}(\mathfrak{g})$, where
$\mathcal{O}^{r,s}_f$ is defined as the category of finite-dimensional weight modules (of type $1$) of $U_{r,s}(\mathfrak{g})$, and proved that
\begin{theorem}\label{Deformation Theory}
	Assume that $rs^{-1}=q^2,$ there is an equivalence as braided tensor categories that
	\begin{align*}		
		\mathcal{O}^{r,s}_f \simeq \mathcal{O}^{q,q^{-1}}_f \simeq \mathcal{O}^{q}_f,
	\end{align*}
	where $\mathcal{O}^{q}_f$ is the category of the finite-dimensional weight modules $($of type $1$$)$ of the quantum group $U_q(\mathfrak{g})$, and the equivalence takes $L(\lambda) \in \mathcal{O}^q_f$ to a deformed weight module in $\mathcal{O}^{r,s}_f$ which is just $L(\lambda)$ defined in Lemma \ref{irrmod}.
\end{theorem}
Combining the results in \cite{E95,D23}:
\begin{theorem}	
	The $K(\breve{U}_q):= {Gr}(\mathcal{O}^{q}_f) \otimes_\mathbb{Z} \mathbb{K},$ the Grothendieck ring $($over $\mathbb{K}$$)$ of the category $\mathcal{O}^q_f$ of the quantum group $\breve{U}_q(\mathfrak{g}),$ is a polynomial algebra. More precisely, let $\{\varpi_i\}_{i=1}^n$ be the set of fundamental weights of $\mathfrak{g}$, then
	\begin{align*}		
		Z(\breve{U}_q) \cong K(\breve{U}_q) = \mathbb{K}[\,[L(\varpi_1)],\cdots,[L(\varpi_n)]\,].
	\end{align*}				
\end{theorem}	
\noindent
with Theorem \ref{Deformation Theory}, we arrive at the following
\begin{theorem}	
	The algebra $K(\breve{U}_{r,s}):= {Gr}(\mathcal{O}^{r,s}_f) \otimes_\mathbb{Z} \mathbb{K}$ is a polynomial algebra. When rank $n$ is even, the centre of the extended two-parameter quantum group $\breve{U}_{r,s}(\mathfrak{g})$ is a polynomial algebra $Z(\breve{U}_{r,s}) = \mathbb{K}[z_{\varpi_1},\cdots,z_{\varpi_n}]$. When rank $n$ is odd, the centre $Z(\breve{U}_{r,s}) \supseteq \mathbb{K}[z_{\varpi_1},\cdots,z_{\varpi_n}] \otimes \mathbb{K}[z_*^{\frac{1}{\ell}}, z_*^{-\frac{1}{\ell}}]$, where $\ell=2$, except $\ell=4$ for $D_{2k+1}$.
\end{theorem}

\begin{remark}
	When $n=\text{rank}\,(\mathfrak g)$ is odd, we still cannot claim if the centre of $U_{r,s}(\mathfrak g)$ or $\breve{U}_{r,s}(\mathfrak g)$ is equal to $\mathbb{K}[z_{\varpi_1},\cdots,z_{\varpi_n}] \otimes \mathbb{K}[z_*^{\frac{1}{\ell}}, z_*^{-\frac{1}{\ell}}]$ or not. It remains an open question.		
\end{remark}

	\typeout{get arXiv to do 4 passes: Label(s) may have changed. Rerun}
\end{document}